\documentclass[11pt,dvipsnames]{amsart}

\usepackage{amsmath,amssymb,amsfonts, mathtools}
\usepackage[mathscr]{eucal}
\usepackage[margin=1.5in]{geometry}

\usepackage{tikz, tikz-cd}%
\usetikzlibrary{matrix,arrows,decorations.pathmorphing,cd,patterns,calc,backgrounds,patterns}
\tikzset{dummy/.style= {circle,fill,draw,inner sep=0pt,minimum size=1.2mm}}
\tikzset{vertex/.style={fill, circle, minimum size=.1cm, inner sep=0pt}}
\usetikzlibrary{decorations.pathreplacing, calc}
\usepackage{extarrows}

\usepackage{marvosym}
\usepackage{quiver} 

\usepackage{graphicx,scalerel}

\usepackage[pagebackref, colorlinks, citecolor=DarkOrchid, linkcolor=PineGreen, urlcolor=NavyBlue]{hyperref}
\hypersetup{linktoc=all,} 

\usepackage{comment}
\usepackage[shortlabels]{enumitem}
\usepackage{hyperref}

\usepackage[textwidth=25mm, textsize=tiny]{todonotes}
\usepackage[final]{microtype}
\usepackage{appendix}

\numberwithin{equation}{section} 
\numberwithin{figure}{section}
\usepackage[nameinlink,capitalise,noabbrev]{cleveref}

\newcommand{\newrefformat}[2]{}

\usepackage{stmaryrd}

\newcommand\restr[2]{{
  \left.\kern-\nulldelimiterspace 
  #1 
  \vphantom{\big|} 
  \right|_{#2} 
  }}

\crefname{lemma}{Lemma}{Lemmas}
\crefname{theorem}{Theorem}{Theorems}
\crefname{definition}{Definition}{Definitions}
\crefname{proposition}{Proposition}{Propositions}
\crefname{remark}{Remark}{Remarks}
\crefname{observation}{Observation}{Observations}
\crefname{construction}{Construction}{Constructions}
\crefname{corollary}{Corollary}{Corollaries}
\crefname{question}{Question}{Questions}
\crefname{equation}{Equation}{Equations}
\crefname{construction}{Construction}{Constructions}
\crefname{ex}{Example}{Examples}
\crefname{appsec}{Appendix}{Appendices}
\crefname{subsection}{Subsection}{Subsections}
\Crefname{warning}{Warning}{Warnings}

\theoremstyle{plain}
\newtheorem{theorem}[equation]{Theorem}

\newtheorem{proposition}[equation]{Proposition}
\newtheorem{lemma}[equation]{Lemma}

\theoremstyle{definition}
\newtheorem{definition}[equation]{Definition}

\newtheorem{remark}[equation]{Remark}
\newtheorem{construction}[equation]{Construction}
\newtheorem{observation}[equation]{Observation}

\usepackage[textwidth=25mm, textsize=tiny]{todonotes}

\newcommand{\THH}{\mathrm{THH}}

\newcommand{\E}{\mathrm{E}}
\newcommand{\Tor}{\mathrm{Tor}}
\newcommand{\hocolim}{\mathrm{hocolim}}

\renewcommand{\k}{{{\Bbbk}}}

\newcommand{\Z}{{\mathbb Z}}

\newcommand{\Q}{{\mathbb Q}}
\newcommand{\F}{{\mathbb F}}
\newcommand{\bbL}{{\mathbb L}}

\newcommand{\B}{\mathfrak{B}}

\tikzcdset{
  diagrams={>={Straight Barb[scale=0.5]}}
}
\tikzset{
  altstackar/.style={decorate, decoration={show path construction,
    lineto code={
      \path (\tikzinputsegmentfirst); \pgfgetlastxy{\xstart}{\ystart}
      \path (\tikzinputsegmentlast); \pgfgetlastxy{\xend}{\yend}
      \path ($(0,0)!1.5pt!(\ystart-\yend,\xend-\xstart)$); \pgfgetlastxy{\xperp}{\yperp}
      \foreach \n[evaluate=\n as \k using .5*#1-\n+.5] in {1,...,#1}{
        \ifodd\n{\draw[->, shorten <=2pt, shift={($\k*(\xperp,\yperp)$)}](\xstart,\ystart)--(\xend,\yend);}
        \else{\draw[<-, shorten >=2pt, shift={($\k*(\xperp,\yperp)$)}](\xstart,\ystart)--(\xend,\yend);}\fi
      }
    }
  }}, altstackar/.default={1}
}

\makeatletter
\def\slashedarrowfill@#1#2#3#4#5{%
  $\m@th\thickmuskip0mu\medmuskip\thickmuskip\thinmuskip\thickmuskip
  \relax#5#1\mkern-7mu%
  \cleaders\hbox{$#5\mkern-2mu#2\mkern-2mu$}\hfill
  \mathclap{#3}\mathclap{#2}%
  \cleaders\hbox{$#5\mkern-2mu#2\mkern-2mu$}\hfill
  \mkern-7mu#4$%
}
\def\rightslashedarrowfill@{%
  \slashedarrowfill@\relbar\relbar\mapstochar\rightarrow}
\newcommand\xslashedrightarrow[2][]{%
  \ext@arrow 0055{\rightslashedarrowfill@}{#1}{#2}}
\makeatother

\newcommand\bbS{\mathbb{S}}
\newcommand\X{\mathbb{Z}[[u_1]]_{\mathbb{Q}}}
\newcommand{\bTHH}{\overline{THH}}
\newcommand{\bbTHH}{\overline{\overline{THH}}}

\begin{document}

\author[Agarwal]{Sanjana Agarwal}
\address{Department of Mathematics, Indiana University Bloomington, 831 E Third St, Bloomington, IN 47405}
\email{sanjagar@iu.edu}

\title[$\THH$ of the Morava E-theory spectrum $E_2$]{$\THH$ of the Morava E-theory spectrum $E_2$}

\begin{abstract} 
The Morava $E$-theories, $E_{n}$, are complex-oriented $2$-periodic ring spectra, with homotopy groups $W_{\F_{p^{n}}}[[u_{1}, u_{2}, ... , u_{n-1}]][u,u^{-1}]$. Here $W$ denotes the ring of Witt vectors. $E_{n}$ is a Landweber exact spectrum and hence uniquely determined by its homotopy groups as $BP_{\ast}$-algebra. Algebraic $K$-theory of $E_{n}$ is a key ingredient towards analyzing the layers in the $p$-complete Waldhausen’s algebraic $K$-theory chromatic tower. One hopes to use the machinery of trace methods to get results towards algebraic $K$-theory once the computation for $THH(E_{n})$ is known. 

In this paper we describe $THH(E_{2})$ as part of consecutive chain of cofiber sequences where each cofiber sits in the next cofiber sequence and the first term of each cofiber sequence is describable completely in terms of suspensions and localizations of $E_{2}$.
For these results, we first calculate $K(i)$-homology of $THH(E_{2})$ using a Bökstedt spectral sequence and then lift the generating classes of $K(1)$-homology to fundamental classes in homotopy group of $THH(E_{2})$. These lifts allow us to construct terms of the cofiber sequence and explicitly understand how they map to $THH(E_{2})$. 
\end{abstract}

\maketitle
\tableofcontents

\section{Introduction}
The goal of this article is to study and describe the topological Hochschild homology of the second Morava $E$-theory spectrum, $THH(E_{2})$, in terms of the suspensions and localizations of $E_{2}$-spectrum (\cref{MainThm}). After $p$-completion, we are able to give a full description of $THH(E_{2})^{\wedge}_p$ in these terms as follows.

\begin{theorem}[\cref{MainCorr}]
We have the following diagram of $(E_2)^{\wedge}_p$-modules for $THH(E_{2})_{p}^{\wedge}$, where $(C_{f_{i}})_{p}^{\wedge}$ are the cofiber maps of $(f_{i})_{p}^{\wedge}$
\[\begin{tikzcd}
(E_{2})_{p}^{\wedge} \arrow{r}{(f_{1})_{p}^{\wedge}:= \text{p-completed unit map}}[swap]{\eqref{unit map}} & THH(E_{2})_{p}^{\wedge} \arrow{d}{(C_{f_{1}})_{p}^{\wedge}}\\
\Sigma^{2p-1}L_{1}(E_{2})_{p}^{\wedge} \arrow{r}{(f_{2})_{p}^{\wedge}:=(\overline{\jmath_{1}})_{p}^{\wedge}}[swap]{\eqref{Define X_2 and f_2}} & \overline{THH}(E_{2})_{p}^{\wedge} \arrow{d}{(C_{f_{2}})_{p}^{\wedge}}\\
(\bigvee_{\alpha} (\Sigma^{|\alpha|}L_{1}E_{2}[u_{1}^{-1}] \bigvee \Sigma^{|\alpha|+2p-1}L_{1}E_{2}[u_{1}^{-1}]))_{p}^{\wedge}\arrow[r, phantom, sloped, "\simeq"]&\bbTHH(E_2)_{p}^{\wedge}.
\end{tikzcd}\]
\end{theorem}

Here $\bTHH(E_2)$ is the cofiber of the unit map $f_1$ and $\bbTHH(E_2)$ is the cofiber of the map of $E_2$-modules, $f_2$. Both cofibers are $L_1$-local where $L_1$ denotes the Bousfield localization with respect to the first Johnson-Wilson theory $E(1)$. For an $E_2$-module $M$, by $M[u_1^{-1}]$ we mean the spectrum $$\hocolim(M\xrightarrow{u_{1}}M\xrightarrow{u_{1}}M\xrightarrow{u_{1}}\cdots)$$ where the $u_{1}$-multiplication is defined by the $E_{2}$-module structure of $M$. See \S5.3 for details on indexing, $\alpha$.

The cofiber sequences in the diagram above, are possible to construct due to the homotopy classes of $THH(E_2)$ that we are able to lift from $K(i)$-homology classes of $THH(E_2)$ along the Hurewicz map. Here $K(i)$ is the $i$th Morava $K$-theory. Computing $K(i)_{\ast}THH(E_2)$ is crucial to our procedure and we show the following:

\begin{theorem}[\cref{K0 homology of THHE_2}, \cref{K(1) of THH(E2)}, \cref{K(2) of THH(E2)}, \cref{higher i}] The spectrum $THH(E_2)$ is $L_2$-local. And, we have following isomorphisms of $K(0)_{\ast}E_2$, $K(1)_{\ast}E_2$ and $K(2)_{\ast}E_2$-algebras, respectively
\begin{align*}
K(0)_{\ast}THH(E_{2}) & \cong (K(0)_{\ast}E_{2}) \otimes_{\X} (\Lambda_{\X}HH_{1}^{\Q}(\X)) \otimes_{\Q} (\Lambda_{\Q}du) \\
K(1)_{\ast}THH(E_{2}) & \cong (K(1)_{\ast}E_{2}) \otimes_{{F}_{p}[[u_{1}]]} (\Lambda_{\mathbb{F}_{p}[[u_{1}]]}HH_{1}^{\F_{p}[u_{1}]}\F_{p}[[u_{1}]]) \otimes_{\F_{p}} (\Lambda_{\F_{p}} dt_{1})\\
K(2)_{\ast} THH (E_{2}) & \cong K(2)_{\ast}E_{2}
\end{align*}
where $\X:= \mathbb{Z}_{p}[[u_{1}]] \otimes_{\Z} \Q$ and $du, dt_1$ are as in \eqref{du classes},\eqref{dt_1 class}.
\end{theorem}
Here, for $M$ an $R$-module, $\Lambda_{R}M$ is the (graded) exterior $R$-algebra on $M$. And, $\Lambda_{\Q}du:= \Lambda_{\Q}\Q\langle du\rangle$, $\Lambda_{\F_p}dt_1:= \Lambda_{\F_p}\F_p\langle dt_1\rangle$.

In this section, we give some context to the question we address in this paper and comment on some unanswered questions and future directions. For terminologies and notations used in this paper, please look at \cref{Notations}.

\subsection{Motivation}
The 2002 paper of Ausoni and Rognes \cite{10.1007/BF02392794} initiated a program to understand algebraic $K$-theory of ring spectra and its relationship to chromatic phenomena. We begin with a brief outline of this program as it has served as one of the big motivations behind many of the calculations in this area over the last 20 years.

In algebraic number theory, the arithmetic of the ring of integers in a number field is largely captured by its Picard group, unit group, etc. These groups are closely related to the algebraic $K$-theory groups associated to this ring: $K_{0}(R)$, $K_{1}(R)$, respectively. Thus, algebraic $K$-theory encodes extremely rich information about the arithmetic structures of rings within itself. 

There is a symmetric monoidal functor from the category of rings to the category of spectra
\begin{align*}
  H(-): \mathscr{R}ing & \rightarrow \mathscr{S}p\\
    R & \mapsto HR
\end{align*}
where $HR$ is the Eilenberg-MacLane spectrum associated to $R$.
 Since we can extend the algebraic $K$-theory functor to ring spectra (see for example \cite[Chapter 6]{MR1417719}), a natural question to ask is: what structural information about ring spectra is encoded by algebraic $K$-theory. Waldhausen in \cite{MR802796} showed that algebraic $K$-theory for ring spectra, particularly the sphere spectrum $\bbS$, is related to high dimensional differential topology. Thus, understanding $K(\bbS)$ is an important and relevant question for many subjects.
 
Work of Dundas, Goodwillie, McCarthy \cite[Theorem 0.0.2]{MR3013261} and Hesselholt, Madsen \cite[Theorem B.1]{MR1410465} shows that the square 
\[ \begin{tikzcd}
K(\bbS) \arrow{r} \arrow{d} & K(\bbS_{p}) \arrow{d}\\
K(\pi_{0}\bbS) \arrow{r} & K(\pi_{0}\bbS_{p})
\end{tikzcd} \]
is homotopy cartesian after $p$-completion ($\bbS_{p}$ is the p-complete sphere spectrum). So up to $p$-completion, it is enough to understand $K(\bbS_{p})$. $\bbS_{p}$ is the homotopy limit of the chromatic tower 
\[ 
\cdots \rightarrow L_{n}\bbS_{p} \rightarrow \cdots \rightarrow L_{1}\bbS_{p} \rightarrow L_{0}\bbS_{p} = H\Q_{p} \]
where $L_{n}$ denotes the Bousfield localization with respect to the $n$th Johnson-Wilson theory $E(n)$ (see \cref{Notations}\eqref{localization},\eqref{Johnson-Wilson theory}). The Hopkins-Ravenel chromatic convergence theorem implies that the map
\[ \bbS_{p} \overset{\simeq} \longrightarrow \text{holim}_{n} \ L_{n}\bbS_{p} \]
is a weak equivalence. It also induces a tower on algebraic $K$-theory
\[ 
\cdots \rightarrow K(L_{n}\bbS_{p}) \rightarrow \cdots \rightarrow K(L_{1}\bbS_{p}) \rightarrow K(L_{0}\bbS_{p}) = K(\Q_{p}) \]
with a map 
$$K(\bbS_{p}) \rightarrow \text{holim}_{n} \ K(L_{n}\bbS_{p}).$$
Waldhausen \cite{10.1007/BFb0075567} conjectured that this map should be a weak equivalence. Waldhausen's work also suggests that there should be an algebraic $K$-theoretic interpretation of the fibers of the maps in this sequence: the fibers should be closely related to $K(L_{K(n)}\bbS_{p})$, where $L_{K(n)}$ is Bousfield localization with respect to $n$th Morava $K$-theory $K(n)$ (see \cref{Notations}\eqref{localization},\eqref{Morava K-theory}). 

From the work of Devinatz-Hopkins \cite{MR2030586}, $L_{K(n)}\bbS_{p}$ is closely related to $E_n$, the $n$th Morava $E$-theory (see \cref{Notations}\eqref{Morava E-theory}) as  
$$L_{K(n)}\bbS_{p} \simeq E_{n}^{h\mathbb{G}_{n}}$$
where $$\mathbb{G}_{n}:= \mathbb{S}_{n} \rtimes C_{n}$$ is the semidirect product of $n$th profinite Morava stabilizer group and the cyclic group of order $n$. Based on this, Rognes conjectured that $$K(L_{K(n)}\bbS_{p}) \rightarrow K(E_{n})^{h\mathbb{G}_{n}}$$
is a `nice' map in sufficiently high dimensions: in particular up to smashing with a finite $p$-local $CW$-spectrum of chromatic type $n+1$, this map is a weak equivalence. Hence, in \cite{10.1007/BF02392794}, the authors lay out a plan to analyze $K(E_{n})$. The theme is to construct ``localization sequences" (cofibration sequences analogous to Quillen's algebraic $K$-theory localization sequence \cite{zbMATH03457106}) in algebraic $K$-theory of connective commutative $\bbS$-algebras and then use the machinery of trace methods \cite{Survey} to compute things using topological cyclic homology, $TC$.

The localization sequences conjectured in \cite{10.1007/BF02392794} towards this project were proven to not hold in \cite{MR3760300}. Blumberg and Mandell in \cite{MR4096617} prove the following localization sequences in $K$-theory, $TC$, and topological Hochschild homology $THH$
\begin{equation}\label{localization sequences}
    \begin{aligned}
      \cdots \rightarrow K(W\F_{p^{n}}[[u_{1}, \cdots, u_{n-1}]]) \rightarrow & K(BP_{n}) \rightarrow K(E_{n}) \rightarrow \Sigma \cdots  \\
      \cdots \rightarrow TC(W\F_{p^{n}}[[u_{1}, \cdots, u_{n-1}]]) \rightarrow T&C(BP_{n}) \rightarrow TC(BP_{n}|E_{n}) \rightarrow \Sigma \cdots
      \\
      \cdots \rightarrow THH(W\F_{p^{n}}[[u_{1}, \cdots, u_{n-1}]]) \rightarrow T&HH(BP_{n}) \rightarrow THH(BP_{n}|E_{n}) \rightarrow \Sigma \cdots
    \end{aligned}
\end{equation}
where $W$ denotes the ring of $p$-typical Witt vectors, and $BP_{n}$ denotes the connective cover of $E_{n}$. This is exceptional in the sense that this provides us with tools for calculating algebraic $K$-theory of non-connective ring spectra using localization sequences. 

To implement the machineries of trace methods and localization sequences towards understanding the relationship between algebraic $K$-theory and the chromatic phenomenon, we therefore, first need to know $THH$ of various spectra involved. The topological Hochschild homologies (and algebraic $K$-theories) of various complex-oriented Landweber theories are thus extremely useful, but not yet very well known except in few cases \cite{MR1209233}, \cite{10.1007/BF02392794}, \cite{Ausoni_2010}, \cite{Ausoni_2012}, \cite{Ausoni_2012Second}, \cite{stonek2020highertopologicalhochschildhomology}. In \cite{MR4071375}, Ausoni and Richter make progress towards calculating $THH(E(n))$ but these are under commutativity assumptions on $E(n)$ spectra. Recent developments in the field (for example \cite{burklund2022chromaticnullstellensatz}, \cite{hahn2022redshiftmultiplicationtruncatedbrownpeterson}, \cite{burklund2023ktheoreticcounterexamplesravenelstelescope}) have given us big insights into this program but explicit structure for algebraic $K$-theory, $TC$, and $THH$ of $BP_n$ and $E_n$ remain completely unknown and interesting for $n\geq 2$.

\subsection{Future questions and remarks}
The computations of this paper result in some further questions:

\begin{itemize}[leftmargin=*]
\item The methodology used for the computation in this paper does not currently extend to higher Morava $E$-theories $E_n, \ n\geq 2$. This is due to the fact that the computations of Hochschild homology of rings $$k[[x_1,\ldots,x_n]]$$ ($k$ any field) remain, to our knowledge, unknown for $n \geq 2$. The methods of \cite{MR1853116}, where $$HH_{\ast}^k(k[[x]])$$ is calculated, do not extend to these rings. Hochschild homology of these rings are of interest to us. 
\item In moving from $E(2)$ to $E_2$, one adds a lot of nice structure. But in terms of homotopy groups for us, it means working with power series rings instead of polynomial rings. Taming the power series ring somehow is necessary to make the result stronger. For this reason, continuous $THH$ as defined in Efimov's recent work \cite{efimov2025localizinginvariantsinverselimits} seems to be an appropriate invariant to work with. In upcoming work with Noah Wisdom, we explore some computations in this direction.
\item The computations of this paper can be useful in the broader Rognes program due to the relationship
\[
\begin{tikzcd}
THH(BP_{n}) \arrow{d} &\\
THH(BP_{n})[u^{-1}] \arrow{r}{\simeq} & THH(E_{n})
\end{tikzcd}
\]
and the fact that $THH(BP_{n})$ and $TC(BP_{n})$ are crucial in the localization sequences \eqref{localization sequences} towards understanding $K(E_{n})$. If we can lift the classes of $THH(E_2)$ to $THH(BP_2)$, we should be able to compute the cofiber $THH(BP_2|E_2)$ since we understand $THH(W\F_{p^{2}}[[u_{1}]])$ from \cite{MR1853116}. Note that in general understanding $THH(W\F_{p^{n}}[[u_{1}, \cdots, u_{n-1}]])$ might require the first bullet point above.
\item Given the understanding of classes $\pi_{\ast}THH(E_2)$ in this paper, one can ask if its possible to understand and say something about the cyclotomic structure of $THH(E_2)$. We have been unable to say anything in that direction yet.
\item We are unable to show complete splitting of cofiber diagrams in \cref{MainThm} and \cref{MainCorr}. In fact, we conjecture that these sequences do not split but we do not have a proof so far. 
\end{itemize}

\subsection{Some terminology and notation}\label{Notations}\indent
\begin{enumerate}[leftmargin=*]
    \item \label{Morava K-theory} We denote the \textit{$i$th Morava $K$-theory spectrum} as $K(i)$. These are complex oriented (but not Landweber exact) spectra with $p$-typical formal group laws and coefficient rings  $$K(i)_{\ast} = \F_p[v_{i}, v_{i}^{-1}] \text{, where } |v_i|=2(p^i-1).$$ 
    \item \label{Johnson-Wilson theory} We denote the \textit{$n$th Johnson-Wilson theory spectrum} as $E(n)$. These are complex oriented spectra with $p$-typical formal group laws and coefficient rings $$E(n)_{\ast} = \Z_{(p)}[v_{1}, \cdots, v_{n}, v_{n}^{-1}] \text{, where } |v_i|=2(p^i-1).$$ They are Landweber exact and hence determined by their coefficient rings as $BP_{\ast}$-algebras. 
    \item \label{Morava E-theory} We denote the \textit{$n$th Morava $E$-theory spectrum} as $E_n$. These are complex oriented spectra with $p$-typical formal group laws and coefficient rings $$E_{n \ast} \cong W\F_{p^{n}}[[u_{1}, \cdots\\, u_{n-1}]][u,u^{-1}] \text{, where } |u|=2, |u_i|=0.$$
     They are Landweber exact and hence determined by their coefficient rings as $BP_{\ast}$-algebras. They are $E_{\infty}$-ring spectra. More details on $BP_{\ast}$-algebra structure is discussed in \cref{Classes}.
    \item\label{localization} We denote by $L_n$, the Bousfield localization with respect to $E(n)$ and by $L_{K(n)}$, the Bousfield localization with respect to $K(n)$. Note that $E_n$ is an $L_n$-local spectrum.
    \item\label{HH,THH} For $k$ a commutative ring and $A$ a $k$-algebra, $N^{cy}_k(A)$ denotes the cyclic bar construction of $A$ over $k$. We also call this the Hochschild complex. The homotopy groups of this simplicial abelian group or the homology groups of the associated chain complex is denoted by $HH_{\ast}^k(A)$, the Hochschild homology of $A$ relative to $k$. Topological Hochschild homology, denoted by $THH$, is the topological analogue of Hochschild homology in the category of spectra $(Sp,\bbS,\wedge)$. It is a simplicial spectrum with homotopy groups denoted as $THH_\ast$. For more details see \cite{Survey},\cite{krause2018lectures}.  
    \item \label{field of fractions} $\Q_p$ denotes the $p$-complete rationals. For a field $k$, $k(t)$ denotes the field of fractions of the polynomial ring $k[t]$, $k[[t]]$ represents the ring of formal power series with coefficients in $k$ and with associated field of fractions denoted as $k((t))$.    
 \end{enumerate}

\subsection*{Outline}
In \cref{Prelims}, we give some preliminary tools and results we need for the calculations in this paper. \cref{K(i) of THHE2} is concerned with computations of the $K(i)$-homologies of $THH(E_2)$. For this, we first compute $HH_{\ast, \ast}^{K(i)_{\ast}}K(i)_{\ast}E_2$ and then use the Bökstedt spectral sequence. Next we lift these homology classes along the Hurewicz map to $\pi_{\ast}THH(E_2)$ in \cref{Lifting classes}. Finally, using the lifted classes in $\pi_{\ast}THH(E_2)$, we construct various cofiber sequences to prove our main result in \cref{Main results}.

\subsection*{Acknowledgements}
This paper would not have been possible without the many insightful conversations and the guidance provided by Michael Mandell. We are also grateful to David Mehrle, Srikanth B. Iyengar, Birgit Richter, Hari Rau-Murthy, Allen Yuan, Nat Stapleton and Noah Riggenbach for helpful discussions. We thank Ayelet Lindenstrauss and the anonymous reviewer for numerous comments and suggestions on the exposition. Finally, we would like to thank Kate Ponto for her support and encouragement.

\section{Preliminaries}\label{Prelims}

In this section, we review some preliminary mathematical results needed in the rest of the paper. These results concern the unit map from a ring to its topological Hoschschild homology (\cref{section unit map}), particular homotopy and Hochschild homology classes of relevant complex oriented spectra (\cref{Classes}), the Bökstedt spectral sequence (\cref{BSS}), the Hochschild homology of power series rings (\cref{HH of H-smooth}), and some results of~\cite{MR4071375} on the Hochschild homology of $K(i)_*E(n)$ (\cref{AR calculations}). We include these to make the paper more self-contained; a reader familiar with this material can skip this section and consult it as needed in the later sections.

\subsection{Unit map} \label{section unit map}
There is a unit map (of $E_{\infty}$-ring spectra)
\begin{equation}\label{unit map} 
   E_2 \to THH(E_2)
\end{equation}
 which we denote by $f_1$ for reasons that will be apparent in \cref{Main results}. This map is induced from the map of ring spectrum 
 $$\bbS \rightarrow THH(E_{2})$$ (which represents the unit element in $\pi_{\ast}THH(E_{2})$) after taking smash product with $E_{2}$ and using the $E_{2}$-module structure on $THH(E_{2})$. Since $E_\infty$-ring structures can be rigidified to commutative ring structures \cite[II.3]{MR1417719}, $THH(E_2)$ is modeled as the geometric realization of a simplicial spectrum and the unit map can also be seen as the inclusion of the $1$-skeleton. 
 
 Since $E_2$ is a commutative ring spectrum, there is a map of commutative $E_2$-algebra spectra in the other direction $$THH(E_2) \rightarrow E_2$$
 given by the multiplication maps $E_2^{\wedge n+1} \to E_2$. The composite map
 $$E_2 \to THH(E_2) \to E_2$$
 is identity and thus we have a splitting of $E_{2}$-modules 
    \begin{equation}\label{splitting of unit map}
        THH(E_2) \simeq E_2 \vee \overline{THH}(E_2)
    \end{equation}
    where $\overline{THH}(E_2)$ denotes the cofiber of \eqref{unit map}. The cofiber inherits the structure of a non-unital commutative $E_2$-algebra. 

\subsection{Classes}\label{Classes}
Throughout this paper, we work with various key homological classes that arise out of the Brown-Peterson spectrum $BP$, $BP_{\ast}BP$-theory and maps of ring spectra 
\begin{equation}\label{BP maps}
\begin{aligned}
&BP \rightarrow K(n),   \\ 
&BP \rightarrow E(m) \rightarrow E_{m}
\end{aligned}
\end{equation}
(for any $n, m$).  We use these maps without further comment to make $K(n)_{\ast}$, $E(m)_{\ast}$, and $E_{m\ast}$ into $BP_{\ast}$-modules and $E_{m\ast}$ into an $E(m)_{\ast}$-module.
More details on these spectra, these maps, and the classes defined below can be found in \cite{MR860042}.

$E(m)$ and $E_{m}$ have canonical choices of formal group laws, i.e., canonical ring maps from $BP$. Thus, there are canonical classes 
\begin{equation}\label{v classes 1}
   v_{i} \in \pi_{\ast}E
\end{equation}
for $E = E(m)$ or $E_m$ which are the images of the classes $v_{i}$ in $BP_{\ast}$. Also, from \eqref{BP maps} there are classes  
\begin{align}
 v_i & \in  \pi_{\ast}(E \wedge E) \cong E_{\ast}E, \label{v classes 2}\\
 t_j & \in \pi_{\ast}(E \wedge E) \cong E_{\ast}E \label{t classes 1}
\end{align}
which are the image classes of $$BP_*BP\cong \Z_{(p)}[v_i,t_j] \to E_{\ast}E.$$
These $v_{i}$s are the left Hurewicz image, $\eta_{L}(v_{i})$.
Note that $|t_j|=|v_j| = 2(p^j-1)$, $t_j$ in $BP_{\ast}BP$ is the image of $v_j\in BP_\ast$ under the right unit map $\eta_R$ and $v_i$ in $BP_{\ast}BP$ is the image of $v_i\in BP_\ast$ under the left unit map $\eta_L$. In the same vein, we have
\begin{align} 
    v_i & \in K(n)_{\ast}E, \label{v classes 3}\\
    t_{j} & \in K(n)_{\ast}E \label{t classes 2}
\end{align}
where \eqref{v classes 3} arises from the image $\eta_{L}(v_{i})$ of $K(n)$.

There is a Hurewicz map 
\[\pi_{\ast}(E \wedge E)\to K(n)_{\ast}(E \wedge E) \cong K(n)_{\ast}E \otimes_{K(n)_{\ast}} K(n)_{\ast}E\]
where the isomorphism holds since $K(n)$ satisfies the K\"unneth theorem. 
The image of the class $t_j$ from \eqref{t classes 1} under this map can be found by looking at the corresponding Hurewicz map for $BP$,
\[ \pi_\ast(BP\wedge BP)\to BP_\ast(BP\wedge BP) \cong BP_\ast BP \otimes_{BP_\ast} BP_\ast BP.\]
Here, the isomorphism holds since $BP_\ast BP$ is flat over $BP_\ast$.
Standard formulae for the Hopf algebroid diagonal and antipode~\cite[A2.1.27]{MR860042} then in principle calculate the image of $t_j$ for all $j$.  In the case $j=1$, we get the image class to be
$$t_1\otimes 1+1\otimes t_1.$$
Thus,
\begin{equation}\label{t_1 class}
    t_1 \in \pi_{\ast}(E \wedge E) \mapsto t_1\otimes 1+1\otimes t_1 \in  K(n)_{\ast}E \otimes_{K(n)_{\ast}} K(n)_{\ast}E.
\end{equation}
The Hochschild complex of $K(n)_{\ast}E$ over $K(n)_{\ast}$ has 
$$K(n)_{\ast}E \otimes_{K(n)_{\ast}} K(n)_{\ast}E$$
as the simplicial degree $1$ term. As an element here, $t_1\otimes 1$ is a degenerate class (it is $s_0(t_1)$ where $s_0$ is the zeroth degeneracy map of the Hochschild complex and $t_1\in K(n)_{\ast}E$ is a simplicial degree $0$ element). Therefore, as classes in the Hochschild complex of $K(n)_{\ast}E$ over $K(n)_{\ast}$ 
\begin{equation}\label{degenerate class}
    t_1\otimes 1+1\otimes t_1 = 1\otimes t_1.
\end{equation}
The class $1\otimes t_1$ is in the kernel of difference of face maps $d_0-d_1$ going down from simplicial level $1$ to simplicial level $0$ in the Hochschild complex. Therefore, it is an element of the first Hochschild homology:
\begin{equation}\label{dt_1 class}
    dt_1\in HH_{(1,2p-2)}^{K(n)_{\ast}}K(n)_{\ast}E
\end{equation} denotes this homological class which is the image of $t_1\in \pi_{2p-2}(E \wedge E)$ as in \eqref{t_1 class}.

The $u, u_{i}$ classes of $E_{m \ast}$ are related to $v_{i}$ classes of $BP_{\ast}$ and $E(m)_{\ast}$ via the second map in \eqref{BP maps}. At the level of homotopy groups, $E(m) \to E_m$ sends 
\begin{equation}\label{Relation of u and v}
    \begin{aligned}
        v_{i} & \mapsto u_{i} \cdot u^{p^{i}-1} \text{, for } i < m\\
        v_{m} & \mapsto u^{p^{m}-1}.
    \end{aligned}
\end{equation}
From the relations given by Hazewinkel formula \cite[B.5, Pg 167-171]{MR1192553} we have the following equation relating $u_1, v_1, t_1$ and $u$
\begin{equation}\label{u,v,t}
    u_{1} = u^{1-p}v_{1} = u^{1-p}pt_{1}.
\end{equation}
The Hurewicz map takes the classes
$u, u_{i} \in E_{m \ast}$ to classes that we denote by the same symbol
\begin{equation}\label{u classes 1}
    u, u_{i} \in K(n)_{\ast}E_m.
\end{equation} 
Under the right Hurewicz map
$$K(n)_{\ast}E_m\to K(n)_{\ast}(E_m\wedge E_m) \cong K(n)_{\ast}E_m \otimes_{K(n)_{\ast}} K(n)_{\ast}E_m$$
we have
\begin{equation}\label{u classes 2}
    \begin{aligned}
        u_i & \mapsto 1 \otimes u_{i}\\
        u & \mapsto  1 \otimes u.
    \end{aligned}
\end{equation}
As in the discussion above \eqref{dt_1 class}, these result in elements of first Hochschild homology, denoted
\begin{equation}\label{du classes}
    du_{i} \in HH_{(1,0)}^{K(n)_{\ast}}K(n)_{\ast}E_{m} \text{ and } du \in HH_{(1,2)}^{K(n)_{\ast}}K(n)_{\ast}E_{m}.
\end{equation}

\subsection{The Bökstedt spectral sequence}\label{BSS}

The Bökstedt spectral sequence is an essential tool that computes the homology of $THH(R)$ from the Hochschild homology of the homology of a ring spectrum $R$ for any homology theory that satisfies the Künneth theorem \cite[Theorem IX.2.9]{MR1417719}. Examples of such homology theories include ordinary homology theory with field coefficients and all Morava $K$-theories $K(i)$. 

For the homology theory $K(i)$ and the ring spectrum $E_2$, the associated Bökstedt spectral sequence has the form
\begin{equation} \label{BSS Equation}
     \E^2_{r,s} = (HH_{r}^{K(i)_{\ast}}(K(i)_{\ast}E_{2}))_{s} \implies K(i)_{r+s}THH(E_{2}). 
\end{equation}
Here $K(i)_{\ast}$ and $K(i)_{\ast}E_{2}$ are graded rings, and so the Hochschild homology has an internal degree ($s$ in \eqref{BSS Equation}) in addition to a homological degree ($r$ in \eqref{BSS Equation}).
To construct the spectral sequence, we use the cyclic bar construction of $THH$ and the fact that smash product commutes with geometric realization 
\[ K(i) \wedge THH(E_{2}) = K(i) \wedge |E_{2}^{\wedge (\bullet+1)}|
\cong |K(i)\wedge E_{2}^{\wedge (\bullet+1)}|.\]
The spectral sequence arises from the simplicial filtration using the Künneth isomorphism, which turns the simplicial filtration of $K(i)_{\ast}THH(E_{2})$ into the Hochschild complex of $K(i)_{\ast}E_{2}$ relative to $K(i)_{\ast}$:
\begin{equation}\label{Künneth}
\pi_{*}(K(i)\wedge (E_{2} \wedge \cdots \wedge E_2))\cong (K(i)_*E_{2}) \otimes_{K(i)_*}\cdots \otimes_{K(i)_*} (K(i)_*E_{2}).
\end{equation}
The $\E^1$-page of the spectral sequence, thus, becomes 
\begin{equation}\label{E1 page BSS}
    \E^1_{r,s} = ((K(i)_{\ast}E_2)^{\otimes_{K(i)_{\ast}}r})_s.
\end{equation}

\eqref{BSS Equation} is a multiplicative spectral sequence of commutative $K(i)_{\ast}E_2$-algebras \cite[Proposition 4.3]{Angeltveit_2005}. The $\E^2$-page differentials are maps
\begin{equation}\label{differentials of BSS}
    \E^2_{r,s} \to \E^{2}_{r-2,s+1}
\end{equation}
and satisfy the Leibniz rule.

\subsection{Hochschild homology of $H$-smooth rings}\label{HH of H-smooth} For our computations, we need to know the relative Hochschild homologies of certain rings that are not étale or smooth over the base ring but do satisfy the condition of $H$-smoothness of Larsen--Lindenstrauss~\cite{MR1853116}. Below we summarize some of the results from \cite{MR1853116} and use them to conclude $H$-smoothness of the ring pairs we are interested in.   

\begin{definition}\label{H-smooth}\cite[Definition 1.1]{MR1853116}
    Let $k$ be a commutative ring and $A$ a commutative $k$-algebra. Then the \textit{pair $(k, A)$ is $H$-smooth} (or \textit{$A$ is $H$-smooth over $k$}) if $HH_{\ast}(A)$ is flat over $k$ and the Hochschild-Kostant-Rosenberg map \cite{HKR}
    $$\lambda_{n}^{A/k}: \Lambda^{n}HH_{1}^{k}(A) \rightarrow HH_{n}^{k}(A)$$
    is an isomorphism for all $n$.
\end{definition}

For an $H$-smooth pair $(k,A)$, we have by definition, the calculation of $HH^k_\ast(A)$ in terms of $HH^k_1(A)$. The main results of \cite{MR1853116} that we need allow for identification of some useful $H$-smooth pairs. The following three results are the ones we need.

\begin{proposition}\label{Base change for H-smooth}\cite[Proposition 1.3]{MR1853116}
$H$-smoothness respects arbitrary base change.
\end{proposition}

\begin{proposition}\cite[Corollary 1.6]{MR1853116}\label{Transitivity of H-smooth}
For $k$ a commutative ring, let $A$ be a commutative $k$-algebra and $B$ a commutative $A$-algebra. Then if $A$ is $H$-smooth over $k$ and $B$ is $H$-smooth over $A$, $B$ is $H$-smooth over $k$.
\end{proposition}

\begin{theorem}\cite[Theorem 1.7]{MR1853116}\label{Excellent rings and H-smooth}
If $A$ is an excellent ring as in \cite[\href{https://stacks.math.columbia.edu/tag/07QS}{Tag 07QS}]{stacks-project}, $\mathfrak{m}$ is a maximal ideal in $A$, and $\hat{A}_{\mathfrak{m}}$ is the completion of $A$ at $\mathfrak{m}$, then $\hat{A}_{\mathfrak{m}}$ is $H$-smooth over $A$.
\end{theorem}

\begin{remark}\label{Excellent rings}
   While the definition of an excellent ring is complicated, in practice most rings are excellent. Some examples of excellent rings include fields, Noetherian complete local rings, $\Z$, Dedekind domains with fraction field of characteristic zero, and finite type ring extensions of any of the above. In particular, $\mathbb{Q}$, $\mathbb{Q}_{p}$, $\mathbb{Z}_{p}[t]$, $\mathbb{F}_{p}$ and $\mathbb{F}_{p}[t]$ are excellent rings. 
\end{remark}

\begin{observation}\label{H-smooth pairs}\indent
\begin{enumerate}[leftmargin=*]
    \item\label{1} From \cref{Excellent rings}, \cref{Excellent rings and H-smooth} and the fact that smooth extensions are by definition $H$-smooth, we conclude that the following pairs are $H$-smooth: $(\mathbb{Q}, \mathbb{Q}_{p})$, $(\mathbb{Q}_{p}, \mathbb{Q}_{p}[t])$, $(\mathbb{Z}_{p}[t], \mathbb{Z}_{p}[[t]])$, $(\mathbb{F}_{p}, \mathbb{F}_{p}[t])$ and $(\mathbb{F}_{p}[t], \mathbb{F}_{p}[[t]])$.
    \item\label{2} Doing a base change on pair $(\mathbb{Z}_{p}[t], \mathbb{Z}_{p}[[t]])$ by tensoring with $\Q$ and using \cref{Base change for H-smooth} we get $(\mathbb{Q}_{p}[t], \mathbb{Z}_{p}[[t]] \otimes \Q)$ is $H$-smooth.
    \item\label{3} Applying \cref{Transitivity of H-smooth} on pairs $(\Q, \Q_p)$, $(\Q_p, \Q_p[t])$, $(\Q_p[t], \Z_{p}[[t]] \otimes \Q)$ we observe $(\Q, \Z_{p}[[t]] \otimes \Q)$ is $H$-smooth. Similarly from $H$-smoothness of $(\F_{p}, \F_{p}[t])$ and $(\F_{p}[t], \F_{p}[[t]])$, we conclude $(\F_{p}, \F_p[[t]])$ is $H$-smooth.
    \end{enumerate}
\end{observation}

The final result below tells us the Hochschild homology of power series rings in one variable.

\begin{proposition}\cite[Proposition 2.5, Corollary 2.3]{MR1853116}\label{HH of power series}
 Let $k$ be a field. Then there is a non-canonical isomorphism
 \[ HH_{\ast}^{k}k[[t]] \cong HH_{\ast}^{k}k[t] \otimes_{k[t]} HH_{\ast}^{k[t]}k[[t]] \cong HH_{\ast}^{k}k[t] \otimes_{k[t]} (k[[t]] \oplus \Lambda_{k((t))}^{\geq 1} V_{k}) \]
 where $V_{k} := HH_{1}^{k(t)}k((t))$ is an infinite dimensional vector space over $k((t))$. In the bracketed term, $k[[t]]$ is in homological degree $0$ and $V_{k}$ is in homological degree $1$.
\end{proposition} 

The latter isomorphism in \cref{HH of power series} holds because from \cite[Corollary 2.3]{MR1853116}
$$HH_{i}^{k[t]}k[[t]] \cong HH_{i}^{k(t)}k((t))$$ 
for $i \geq 1$ and thus 
\begin{equation}\label{power series vs field of fractions}
    \Lambda_{k[[t]]}^{\geq 1}HH_{1}^{k[t]}k[[t]] \cong \Lambda_{k((t))}^{\geq 1}HH_{1}^{k(t)}k((t)).
\end{equation}
We will be applying this result mostly in the specific case of $k = \F_{p}$.

\subsection{Hochschild homology calculations for $E(n)$}\label{AR calculations}
Our later computations rely on some of the calculations of Ausoni--Richter~\cite{MR4071375} on the Hochschild homology of the Morava $K$-theory of Johnson-Wilson spectra, $HH^{K(i)_\ast}_\ast(K(i)_\ast E(n))$. While a lot of \cite{MR4071375} requires the hypothesis that the Johnson-Wilson theories are commutative, these computations do not need this hypotheses and so hold unconditionally.  We state the results here in the slightly reformulated form using discussions in the proof of \cite[lemma 5.1]{MR4071375}.

\begin{theorem}\cite[Proposition 2.3, Theorem 3.4, Page 380 line 3]{MR4071375} \label{HH of E(2)} We have the following isomorphisms of $K(0)_{\ast}E(2)$, $K(1)_{\ast}E(2)$ and $K(2)_{\ast}E(2)$-algebras, respectively
\begin{align*}
    HH_{\ast, \ast}^{K(0)_{\ast}}K(0)_{\ast}E(2) &\cong K(0)_{\ast}E(2) \otimes_{\Q} \Lambda_{\Q}(dt_{1}, dt_{2}),\\
HH_{\ast, \ast}^{K(1)_{\ast}}K(1)_{\ast}E(2) &\cong K(1)_{\ast}E(2) \otimes_{\F_p} \Lambda_{\F_{p}}(dt_{1}),\\
HH_{\ast, \ast}^{K(2)_{\ast}}K(2)_{\ast}E(2) &\cong K(2)_{\ast}E(2).
\end{align*}
Here the classes $dt_{1}$ in $$HH_{(1, 2p-2)}^{K(0)_{\ast}}K(0)_{\ast}E(2) \text{ and } HH_{(1, 2p-2)}^{K(1)_{\ast}}K(1)_{\ast}E(2)$$ are induced from the class $t_{1} \in \pi_{2p-1}(E(2) \wedge E(2))$ and the class $$dt_{2} \in HH_{(1, 2p^{2}-2)}^{K(0)_{\ast}}K(0)_{\ast}E(2)$$ is induced from $t_{2} \in \pi_{2p^{2}-1}(E(2) \wedge E(2))$. Also, for $i > 2$, $K(i)_{\ast}E(2)$ is trivial. 
\end{theorem}

\section{$K(i)$ homology groups of $\THH(E_2)$}\label{K(i) of THHE2}

In this section, we calculate $K(i)_\ast(THH(E_2))$ for all $i$ (\cref{K0 homology of THHE_2}, \cref{K(1) of THH(E2)}, \cref{K(2) of THH(E2)} and \cref{higher i}). First, we calculate the Hochschild homology $HH^{K(i)_\ast}(K(i)_\ast E_2)$ of $K(i)_{\ast}E_{2}$ relative to $K(i)_{\ast}$ using the $H$-smoothness results and Ausoni--Richter~\cite{MR4071375} calculations reviewed above.  We then use the Bökstedt spectral sequence \eqref{BSS Equation} to compute $K(i)_{\ast}THH(E_{2})$.

\bigskip
\subsection{The case $i=0$: {$K(0)_{\ast}THH(E_{2})$}}

Note that 
$K(0)_{\ast}E_{2} \cong \pi_{\ast}(E_2)_{\Q}$
and
\begin{equation}\label{rationalization}
    HH_{\ast}^{K(0)_{\ast}}(K(0)_{\ast}E_{2}) \cong HH_{\ast}^{\mathbb{Q}}(\pi_{\ast}(E_{2})_{\mathbb{Q}}).
\end{equation}
To compute \eqref{rationalization}, we need the homotopy groups of rationalized spectrum $E_2$. These are given by tensoring the homotopy groups of $E_2$ with $\Q$. We get
\begin{align*}
    \pi_{\ast}(E_{2})_{\mathbb{Q}} & \cong (\Z_{p}[\mu_{p^{2}-1}][u, u^{-1}][[u_{1}]])\otimes_{\Z} \Q \\
    & \cong (\Z[\mu_{p^{2}-1}]\otimes_{\Z} \Z[u, u^{-1}] \otimes_{\Z} \Z_p[[u_{1}]])\otimes \Q \\
    & \cong (\mathbb{Z}_{p}[[u_{1}]] \otimes_{\Z} \Q) \otimes_{\mathbb{Q}} \mathbb{Q}[\mu_{p^{2}-1}]\otimes_{\mathbb{Q}} \mathbb{Q}[u,u^{-1}].
\end{align*}
Here $\mu_{p^{2}-1}$ denotes the group of $(p^{2}-1)$ roots of unity. For brevity, we denote
\[\X:= \mathbb{Z}_{p}[[u_{1}]] \otimes_{\Z} \Q\]
in the rest of the paper. Next, taking the Hochschild homology with respect to $\Q$, we get
\begin{equation}\label{HH of rational E_2}
   HH_{\ast}^{\mathbb{Q}}(\pi_{\ast}(E_{2})_{\mathbb{Q}}) \cong HH_{\ast}^{\mathbb{Q}}(\X) \otimes_{\mathbb{Q}} HH_{\ast}^{\mathbb{Q}}\mathbb{Q}[\mu_{p^{2}-1}] \otimes_{\mathbb{Q}} HH_{\ast}^{\mathbb{Q}}\mathbb{Q}[u,u^{-1}]. 
\end{equation}

We now identify each of these Hochschild homology terms.

\begin{lemma}\label{HH of Q[u pm1]}
    $HH_{\ast}^{\mathbb{Q}}(\mathbb{Q}[u,u^{-1}]) \cong \mathbb{Q}[u,u^{-1}] \otimes_{\Q} \Lambda_{\mathbb{Q}}(du).$
\end{lemma}

\begin{proof}
    This is due to étaleness of $\Q[u,u^{-1}]$ over $\mathbb{Q}[u]$ which from \cite[E.1.1.8]{MR1600246} implies
    \begin{equation}\label{étale}
       HH_{\ast}^{\mathbb{Q}}(\mathbb{Q}[u,u^{-1}]) \cong \mathbb{Q}[u,u^{-1}] \otimes_{\mathbb{Q}[u]} HH_{\ast}^{\mathbb{Q}}\mathbb{Q}[u], 
    \end{equation}
    and to smoothness of $\mathbb{Q}[u]$ over $\mathbb{Q}$ which from the theorem of Hochschild-Kostant-Rosenberg \cite{HKR}, \cite[Example 2.5]{krause2018lectures} implies
    \begin{equation}\label{smooth}
        HH_{\ast}^{\Q}\Q[u] \cong \Lambda_{\Q[u]}(du) \cong \Q[u] \otimes_{\Q} \Lambda_{\Q}(du).\qedhere
    \end{equation}
\end{proof}

\begin{lemma}\label{HH of Q[mu]}
$HH_{\ast}^{\mathbb{Q}}(\mathbb{Q}[\mu_{p^{2}-1}]) \cong \mathbb{Q}[\mu_{p^{2}-1}].$
\end{lemma}

\begin{proof}
    This is from étaleness of $\mathbb{Q}[\mu_{p^{2}-1}]$ over $\mathbb{Q}$ and from \cite[E.1.1.8]{MR1600246}. We have
\begin{align*}
    HH_{\ast}^{\mathbb{Q}}(\mathbb{Q}[\mu_{p^{2}-1}]) & \cong
\mathbb{Q}[\mu_{p^{2}-1}] \otimes_{\Q} HH_{\ast}^{\mathbb{Q}} \mathbb{Q} \\
& \cong \mathbb{Q}[\mu_{p^{2}-1}].\qedhere
\end{align*}
\end{proof}

\begin{lemma}\label{HH of remaining term}
    $HH_{\ast}^{\Q}(\X) \cong \Lambda_{\X}HH_{1}^{\Q}(\X).$
\end{lemma}

\begin{proof}
    This is from \cref{H-smooth pairs}(\ref{3}) and \cref{H-smooth}.
\end{proof}

Using \cref{HH of Q[mu]}, \cref{HH of Q[u pm1]} and \cref{HH of remaining term}, we rewrite \eqref{HH of rational E_2} as follows.

\begin{proposition}\label{HH of K(0)E2}
 \[HH_{\ast}^{K(0)_{\ast}}(K(0)_{\ast}E_{2})\\ 
 \cong K(0)_{\ast}E_{2} \otimes_{\X} \Lambda_{\X}HH_{1}^{\Q}(\X) \otimes_{\Q} \Lambda_{\Q}(du).\]
 Here, elements in $K(0)_{\ast}E_{2}$ have homological degree $0$, $du$ has homological degree $1$ and internal degree $2$ \eqref{du classes} and elements in $HH_{1}^{\Q}(\X)$ have homological degree $1$ and internal degree $0$.
\end{proposition}
\begin{proof}
\begin{align*}
 &HH_{\ast}^{K(0)_{\ast}}(K(0)_{\ast}E_{2})\\ 
 &\cong \Lambda_{\X}HH_{1}^{\Q}(\X) \otimes_{\Q} \mathbb{Q}[\mu_{p^{2}-1}] \otimes_{\Q} \mathbb{Q}[u,u^{-1}] \otimes_{\Q} \Lambda_{\mathbb{Q}}(du)  \\
 & \cong (\X \otimes \Q[\mu_{p^{2}-1}] \otimes \Q[u,u^{-1}]) \otimes_{\X} \Lambda_{\X} HH_{1}^{\mathbb{Q}}(\X) \otimes_{\Q} \Lambda_{\Q}(du)\\
 & \cong K(0)_{\ast}E_{2} \otimes_{\X} \Lambda_{\X}HH_{1}^{\Q}(\X) \otimes_{\Q} \Lambda_{\Q}(du).
\end{align*}
\end{proof}

This then leads to the computation of $K(0)_{\ast}THH(E_{2})$.

\begin{theorem}\label{K0 homology of THHE_2}
We have the following isomorphism of $K(0)_{\ast}E_{2}$-algebras
\[K(0)_{\ast}THH(E_{2}) \cong K(0)_{\ast}E_{2} \otimes_{\X} \Lambda_{\X}HH_{1}^{\Q}(\X) \otimes_{\Q} \Lambda_{\Q}du.\]
\end{theorem}
\begin{proof}
\cref{HH of K(0)E2} gives us the $\E^2$ page of the Bökstedt spectral sequence \eqref{BSS Equation}. The multiplicative generators lie in homological degree (column) $1$. From the discussion in \eqref{differentials of BSS}, we conclude there are no non-zero differentials for degree reasons and the spectral sequence collapses on the $\E^2$ page.

Since the $\E^\infty$ page is free as a graded commutative $K(0)_{\ast}E_2$ algebra, there are no possible additive or multiplicative extensions. We get an isomorphism of $K(0)_{\ast}E_{2}$-algebras
\[ K(0)_{\ast}THH(E_{2}) \cong K(0)_{\ast}E_{2} \otimes_{\X} \Lambda_{\X}HH_{1}^{\Q}(\X) \otimes_{\Q} \Lambda_{\Q}du.\qedhere\] 
\end{proof}

\bigskip
\subsection{The case $i=1$ and $i=2$: $K(1)_{\ast}THH(E_{2})$ and $K(2)_{\ast}THH(E_2)$}

The calculations for $i=1$ and $i=2$ use similar techniques so we present them together.

Since
\begin{equation*}
    K(i)_{\ast}E_2 \cong E_{2 \ast}K(i)
    \end{equation*}
and from Landweber exactness of $E_2$ 
\begin{equation*}
    E_{2 \ast}K(i) \cong E_{2 \ast}\otimes_{BP_\ast}BP_\ast
 (K(i))
 \end{equation*}
we have $$K(i)_{\ast}E_2 \cong K(i)_{\ast}(BP) \otimes_{BP_{\ast}} (E_{2})_{\ast}.$$
Using the fact that $E_{2\ast}$ is flat over $E(2)_\ast$ and the fact that $p=0$ in $K(i)_*$, the second map in \eqref{BP maps} induces isomorphisms
\begin{equation}\label{K(i) of E2}
K(i)_{\ast}E_{2} \cong K(i)_{\ast}E(2) \otimes_{E(2)_{\ast}} E_{2{\ast}} \cong K(i)_{\ast}E(2) \otimes_{E(2)_{\ast}/p} E_{2{\ast}}/p
\end{equation}
for $i = 1, 2$. This gives 
\begin{equation}\label{HH of K(i)E2}
    HH_{\ast}^{K(i)_{\ast}}(K(i)_{\ast}E_{2}) \cong HH_{\ast}^{K(i)_{\ast}}(K(i)_{\ast}E(2) \otimes_{E(2)_{\ast}/p} E_{2\ast}/p).
\end{equation}

\begin{observation}\label{Steps E(2) to E2}
    The map $E(2)_{\ast}/p \to E_{2 \ast}/p$ induced from the second map in \eqref{BP maps} is flat. At the $\pi_{\ast}$-level, one gets $E_{2 \ast}/p$ from $E(2)_{\ast}/p$ by doing three things on the usual presentations \indent
\begin{enumerate}[leftmargin=*]
    \item adding a $(p^{2}-1)$th root of unity, $\mu_{p^2-1}$,
    \item adding $u$, the $(p^{2}-1)$th root of $v_{2}$ \eqref{Relation of u and v}, and using the generators  $u, u_1$,
    \item completing at the ideal $(u_{1})$.
\end{enumerate}
\end{observation}

Next, we need the computation of the right hand side of \eqref{HH of K(i)E2} in terms of the Hochschild homology of the $\F_p$-algebras $E(2)_{\ast}/p$ and $E_{2 \ast}/p$.  The following lemma is proved in several steps.

\begin{lemma} \label{HH split} For $i = 1,2$
    \[HH_{\ast}^{K(i)_{\ast}}(K(i)_{\ast}E(2) \otimes_{E(2)_{\ast}/p} E_{2 \ast}/p) \cong HH_{\ast}^{K(i)_{\ast}}(K(i)_{\ast}E(2)) \otimes_{HH_{\ast}^{\mathbb{F}_{p}}(E(2)_{\ast}/p)} HH_{\ast}^{\mathbb{F}_{p}}(E_{2 \ast}/p).\]
\end{lemma}

The main step we need is the following theorem. \cref{HH(E2) flat over HH(E(2))} below proves that the hypotheses of \cref{HH and flatness} apply in our case.

\begin{theorem}\label{HH and flatness}
Let $k \rightarrow R \rightarrow S \rightarrow A$ and $k \rightarrow S \rightarrow B$ be maps of commutative rings. Regard $A \otimes_S B$ as an $R$ algebra via the map $R \cong R \otimes_S S\to A\otimes_S B$. If $B$ is flat over $S$ and $HH_{\ast}^{k}(B)$ is flat over $HH_{\ast}^{k}(S)$, then $$HH_{\ast}^{R}(A \otimes_{S} B) \cong HH_{\ast}^{R}(A) \otimes_{HH_{\ast}^{k}(S)} HH_{\ast}^{k}(B).$$
\end{theorem}

\begin{proof}
We use the notation of cyclic bar complex as mentioned in \cref{Notations},\eqref{HH,THH}. Note that as simplicial abelian groups
\begin{equation}\label{cyclic nerve of tensor}
   (N^{cy}_{\otimes_{R}})_{\bullet}(A \otimes_{S}B) \cong (N^{cy}_{\otimes_{R}})_{\bullet}(A) \otimes_{(N^{cy}_{\otimes_{k}})_{\bullet}(S)} (N^{cy}_{\otimes_{k}})_{\bullet}(B). 
\end{equation}
This is a quick check, one sees that the natural map at each simplicial level 
$$(A\otimes_S B)\otimes_R\ldots \otimes_R(A\otimes_S B) \to (A \otimes_R\ldots \otimes_R A)\otimes_{(S \otimes_k\ldots \otimes_k S)}(B \otimes_k\ldots \otimes_k B)$$
is an isomorphism and that these maps respect simplicial structure.

Since $B$ is flat over $S$, the right hand side in \eqref{cyclic nerve of tensor} represents the derived tensor product in simplicial modules $$(N^{cy}_{\otimes_{R}})_{\bullet} (A) \otimes^{\bbL}_{(N^{cy}_{\otimes_{k}})_{\bullet}(S)} (N^{cy}_{\otimes_{k}})_{\bullet} (B)$$ and hence also the derived tensor product in dg-modules 
\begin{equation}\label{cyclic nerve dg modules}
N^{cy}_{\otimes_{R}} (A) \otimes^{\bbL}_{N^{cy}_{\otimes_{k}} (S)} N^{cy}_{\otimes_{k}} (B).    
\end{equation}
Using this formula and taking homotopy groups on both sides of \eqref{cyclic nerve of tensor}, we see
\begin{equation}\label{HH of tensor}
HH_{\ast}^{R}(A\otimes_{S}B) \cong H_{\ast}(N^{cy}_{\otimes_{R}}(A) \otimes^{\bbL}_{N^{cy}_{\otimes_{k}}(S)} N^{cy}_{\otimes_{k}}(B)).
\end{equation}
Using \cref{EMSS} below with $M, N$ and $C$ being $N^{cy}_{\otimes_{R}}(A)$, $N^{cy}_{\otimes_{k}}(B)$ and $N^{cy}_{\otimes_{k}}(S)$ respectively, we have a spectral sequence
\[ \Tor^{HH_{\ast}^{k}S}_{\ast, \ast} (HH_{\ast}^{R}A, HH_{\ast}^{k}B) \Longrightarrow H_{\ast}(N^{cy}_{\otimes_{R}}(A) \otimes^{\bbL}_{N^{cy}_{\otimes_{k}}(S)} N^{cy}_{\otimes_{k}}(B)). \]
Since $HH_{\ast}^{k}(B)$ is flat over $HH_{\ast}^{k}(S)$, all the higher $\Tor$-terms vanish and the spectral sequence collapses, giving us
\[HH_{\ast}^{R}A \otimes_{HH_{\ast}^{k}S} HH_{\ast}^{k}B \cong H_{\ast}(N^{cy}_{\otimes_{R}}(A) \otimes_{N^{cy}_{\otimes_{k}}(S)} N^{cy}_{\otimes_{k}}(B)). \qedhere \]
\end{proof}

We need an algebraic analogue of the Eilenberg-Moore spectral sequence as in \cite[IV.6]{MR1417719} which is a consequence of Tor spectral sequence. We are using the same proof techniques as in \cite[IV.4.1]{MR1417719} and \cite[Section 6]{Lewis_2006}. We expect this spectral sequence is known to experts but we could not find it in literature.

\begin{lemma}\label{EMSS}
Let $C$ be a simplicial or dg-ring, $M$ a right $C$-module, and $N$ a left $C$-module. Then there exists an algebraic Eilenberg-Moore spectral sequence
\[Tor^{H_{\ast}C}_{p,q} (H_{\ast}M, H_{\ast}N) \Longrightarrow H_{p+q}(M \otimes_{C}^{\mathbb{L}}N).\]
\end{lemma}
\begin{proof}
Let $Y$ be a cofibrant approximation of $N$ so that $M \otimes_{C} Y$ represents the derived tensor product $M \otimes^{\bbL}_{C} N$. Choose a free $H_{\ast}C$ resolution of $H_{\ast}Y$
\[ \cdots \rightarrow F_{n} \xrightarrow{\overline{d_{n}}} \cdots \rightarrow F_{1} \xrightarrow{\overline{d_{1}}} F_{0} \xrightarrow{\overline{d_{0}}} H_{\ast}Y. \]
We can realize $\overline{d_{0}}$ as $H_{\ast}$ of a map of $C$-modules $$\bigoplus_{\text{gens of} \ F_{0}} \Sigma^{n_{\alpha_{0}}}C \xrightarrow{d_{0}} Y$$
where $F_{0}$ is graded, $n_{\alpha}$ is the internal degree, and the direct sum is over a chosen set of generators of $F_0$. Let $Y_{-1} = Y$ and $Y_{0}$ be the homotopy cofiber of the map $d_{0}$, where the cofiber map $Y_{-1} \to Y_0$ is denoted by $C_{d_0}$. We have
\begin{equation}\label{cofiber sequence}
   \bigoplus_{\text{gens of} \ F_{0}} \Sigma^{n_{\alpha_{0}}}C \xrightarrow{d_{0}} Y_{-1} \xrightarrow{C_{d_0}} Y_0 \to \Sigma (\bigoplus_{\text{gens of} \ F_{0}} \Sigma^{n_{\alpha_{0}}}C)
\end{equation}
and using the induced long exact sequence in homology, we have 
\begin{equation}\label{les}
    \cdots \to F_0 \xrightarrow[]{\overline{d_0}} H_{\ast}Y_{-1} \xrightarrow[]{\overline{C_{d_0}}} H_{\ast}Y_0 \to \Sigma F_0 \to \cdots.
\end{equation}
From the surjectivity of $\overline{d_{0}}$ on homology, we get 
$$\overline{C_{d_0}} = 0 \text{ and } H_{\ast}(Y_{0}) \cong \Sigma ker(\overline{d_{0}}),$$ the suspension of the kernel module. Next we realize $\Sigma \overline{d_{1}}$, similarly, as a map of $C$-modules $$\bigoplus_{\text{gens of} \ F_{1}} \Sigma^{n_{\alpha_{1}}+1}C \xrightarrow{\Sigma d_{1}} Y_{0}$$ which is surjective on homology. Let $Y_{1}$ be the cofiber of $\Sigma d_1$, and $C_{d_1}$ be the cofiber map from $Y_0$ to $Y_1$. Then, as above,
$$\overline{C_{d_1}} = 0 \text{ and } H_{\ast}(Y_{1}) \cong \Sigma^2 ker(\overline{d_{1}}).$$
Iterating this process gives a filtration
\[ Y_{-1} \xrightarrow[]{C_{d_0}} Y_{0} \xrightarrow[]{C_{d_1}} Y_{1} \to \cdots.\]
Let $Y_{\infty} = \hocolim \ Y_{i}$.  The maps in this system are zero on homology so $Y_{\infty}$ is quasi-isomorphic to $0$. Tensoring this filtration by $X$ over $N$, $X \otimes_{N} Y_{\infty}$ is also quasi-isomorphic to 0 and the filtration on $X \otimes_{N} Y_{\infty}$ gives a spectral sequence converging to 
$$H_{\ast}(X \otimes Y_{-1}) \cong M \otimes^{\bbL}_{C} N$$
where the associated graded is $X \otimes (\bigoplus \Sigma^{n_{\alpha}}N),$ the $\E^{1}$-term is $H_{\ast}X \otimes_{H_{\ast} N} F_{i}$ and the $\E^{2}$-term is the required $\Tor$ term.
\end{proof}

The next statement says that \cref{HH and flatness} can be applied on $$HH_{\ast}^{K(i)_{\ast}}(K(i)_{\ast}E(2) \otimes_{E(2)_{\ast}/p} E_{2 \ast}/p).$$

\begin{theorem}\label{HH(E2) flat over HH(E(2))}
$HH_{\ast}^{\mathbb{F}_{p}}(E_{2 \ast}/p)$ is flat over $HH_{\ast}^{\mathbb{F}_{p}}(E(2)_{\ast}/p)$. Moreover, $HH_{\ast}^{\mathbb{F}_{p}}(E_{2 \ast}/p)$ is isomorphic to
\[\bigg(HH_{\ast}^{\mathbb{F}_{p}}(E(2)_{\ast}/p) \otimes_{\mathbb{F}_{p}[v_{1}][v_{2},v_{2}^{-1}]} (\mathbb{F}_{p}[\mu_{p^{2}-1}][u_{1}][u, u^{-1}])\bigg) \otimes_{\F_{p}[u_{1}]} \bigg(\Lambda_{\mathbb{F}_{p}[[u_{1}]]}HH_{1}^{\F_{p}[u_{1}]}\F_{p}[[u_{1}]]\bigg).\]
\end{theorem}

\begin{proof}
Note
\begin{align}\label{Mod p homotopy}
   E(2)_{\ast}/p \cong \mathbb{F}_{p}[v_{1}][v_{2}, v_{2}^{-1}], \text{ and } E_{2 \ast}/p \cong \mathbb{F}_{p}[\mu_{p^{2}-1}][[u_{1}]][u, u^{-1}]. 
\end{align}
From \cref{Steps E(2) to E2} we have the following sequence of labelled maps
\[\mathbb{F}_{p}[v_{1}][v_{2}, v_{2}^{-1}] \xrightarrow{1} \mathbb{F}_{p}[\mu_{p^{2}-1}][v_{1}][v_{2}, v_{2}^{-1}] \xrightarrow{2} \mathbb{F}_{p}[\mu_{p^{2}-1}][u_{1}][u, u^{-1}] \xrightarrow{3} \mathbb{F}_{p}[\mu_{p^{2}-1}][[u_{1}]][u, u^{-1}].\]
Map $1$ is unramified and flat, hence étale. From \cite[E.1.1.8]{MR1600246}, we have
\begin{equation}\label{HH eqn1}
\begin{aligned}
& HH_{\ast}^{\F_{p}}(\mathbb{F}_{p}[\mu_{p^{2}-1}][v_{1}][v_{2}, v_{2}^{-1}])\\ 
& \quad \cong HH_{\ast}^{\F_{p}}(\mathbb{F}_{p}[v_{1}][v_{2}, v_{2}^{-1}]) \otimes_{\mathbb{F}_{p}[v_{1}][v_{2}, v_{2}^{-1}]} \mathbb{F}_{p}[\mu_{p^{2}-1}][v_{1}][v_{2}, v_{2}^{-1}]\\
& \quad \cong HH_{\ast}^{\F_{p}}(\mathbb{F}_{p}[v_{1}][v_{2}, v_{2}^{-1}]) \otimes_{\F_p} \mathbb{F}_{p}[\mu_{p^{2}-1}].
\end{aligned}
\end{equation}
We conclude 
\begin{equation}\label{map 1 flat}
    HH_{\ast}^{\F_{p}}(\mathbb{F}_{p}[v_{1}][v_{2}, v_{2}^{-1}]) \rightarrow HH_{\ast}^{\F_{p}}(\mathbb{F}_{p}[\mu_{p^{2}-1}][v_{1}][v_{2}, v_{2}^{-1}])
\end{equation}
is flat.

From \eqref{Relation of u and v}, map $2$ sends 
$$v_{1} \mapsto u_{1}u^{p-1} \text{ and } v_{2} \mapsto u^{p^{2}-1}.$$ 
In other words, its attaching a ($p^{2}-1$)-root of a unit and replacing a variable with a unit times that variable. Hence map $2$ is étale and gives
\begin{equation}\label{HH eqn2}
\begin{aligned}
& HH_{\ast}^{\F_{p}}(\mathbb{F}_{p}[\mu_{p^{2}-1}][u_{1}][u,u^{-1}]) \\
& \quad \cong HH_{\ast}^{\F_{p}}(\mathbb{F}_{p}[\mu_{p^{2}-1}][v_{1}][v_{2},v_{2}^{-1}]) \otimes_{\mathbb{F}_{p}[\mu_{p^{2}-1}][v_{1}][v_{2},v_{2}^{-1}]} \mathbb{F}_{p}[\mu_{p^{2}-1}][u_{1}][u,u^{-1}]\\
& \quad \cong HH_{\ast}^{\F_{p}}(\mathbb{F}_{p}[\mu_{p^{2}-1}][v_{1}][v_{2},v_{2}^{-1}]) \otimes_{\mathbb{F}_{p}[\mu_{p^{2}-1}][v_{1}][v_{2},v_{2}^{-1}]} \mathbb{F}_{p}[\mu_{p^{2}-1}][v_{1}][u, u^{-1}].
\end{aligned}
\end{equation}
Since $u$ is a $p^2-1$ root of the image of $v_2$, this is a free module extension of 
\begin{equation}\label{flat 2}
  HH_{\ast}^{\F_{p}}(\mathbb{F}_{p}[\mu_{p^{2}-1}][v_{1}][v_{2},v_{2}^{-1}])  
\end{equation}
and hence flat.

Map $3$ is the completion at ideal $(u_{1})$ and hence flat. Taking Hochschild homology relative to $\F_p$ on both sides of this map, we have 
\begin{equation}\label{HH on map3}
HH_{\ast}^{\F_{p}}(\mathbb{F}_{p}[u_{1}]) \otimes HH_{\ast}^{\F_{p}}(\mathbb{F}_{p}[\mu_{p^{2}-1}][u, u^{-1}]) \rightarrow HH_{\ast}^{\F_{p}}(\mathbb{F}_{p}[[u_{1}]]) \otimes HH_{\ast}^{\F_{p}}(\mathbb{F}_{p}[\mu_{p^{2}-1}][u, u^{-1}])
\end{equation}
where the tensors are over $\F_p$.
Applying \cref{HH of power series} 
\begin{equation}\label{HH of Fp power series}
   HH_{\ast}^{\mathbb{F}_{p}}\mathbb{F}_{p}[[u_{1}]] \cong (HH_{\ast}^{\mathbb{F}_{p}}\mathbb{F}_{p}[u_{1}]) \otimes_{\mathbb{F}_{p}[u_{1}]} (\F_{p}[[u_{1}]] \oplus \Lambda^{\geq 1}_{\mathbb{F}_{p}[[u_{1}]]}V_{\F_{p}}) 
\end{equation}
where $V_{\F_{p}}$ is a vector space over $\F_{p}((u_{1}))$ and $\F_{p}((u_{1}))$ is a localization of $\F_{p}[[u_{1}]]$, we conclude
\begin{equation}\label{flat 3}
    HH_{\ast}^{\mathbb{F}_{p}}\mathbb{F}_{p}[u_{1}] \rightarrow (HH_{\ast}^{\mathbb{F}_{p}}\mathbb{F}_{p}[u_{1}]) \otimes_{\mathbb{F}_{p}[u_{1}]} (\F_{p}[[u_{1}]] \oplus \Lambda^{\geq 1}_{\mathbb{F}_{p}[[u_{1}]]}V_{\F_{p}})
\end{equation}
is a map into an infinite vector space of a localization and thus flat. 

From \eqref{map 1 flat}, \eqref{flat 2} and \eqref{flat 3}, $HH_{\ast}^{\mathbb{F}_{p}}(E_{2 \ast}/p)$ is flat over $HH_{\ast}^{\mathbb{F}_{p}}(E(2)_{\ast}/p)$. In fact,

\begin{align*}
    & HH_{\ast}^{\mathbb{F}_{p}} \mathbb{F}_{p}[\mu_{p^{2}-1}][[u_{1}]][u, u^{-1}]\\
    & \qquad \cong HH_{\ast}^{\mathbb{F}_{p}} \mathbb{F}_{p}[\mu_{p^{2}-1}][u, u^{-1}] \otimes_{\F_p} HH_{\ast}^{\mathbb{F}_{p}}\mathbb{F}_{p}[[u_{1}]]\\
    & \qquad \overset{\eqref{HH of Fp power series}}{\cong} HH_{\ast}^{\mathbb{F}_{p}} \mathbb{F}_{p}[\mu_{p^{2}-1}][u_{1}][u, u^{-1}] \otimes_{{F}_{p}[u_{1}]} \Lambda_{\mathbb{F}_{p}[[u_{1}]]}HH_{1}^{\F_{p}[u_{1}]}\F_{p}[[u_{1}]]
\end{align*}
 Applying \eqref{HH eqn1} and \eqref{HH eqn2} on the last term, we have the result.
\end{proof}

From \cref{HH(E2) flat over HH(E(2))} and \cref{HH and flatness}, we see \cref{HH split} holds.

Finally, we compute $K(1)$ and $K(2)$ homologies of $THH(E_2)$ in the results below.
\begin{theorem}\label{K(1) of THH(E2)}
We have the following isomorphism of $K(1)_{\ast}E_{2}$-algebras
\[K(1)_{\ast}THH(E_{2}) \cong K(1)_{\ast}E_2 \otimes_{{\F}_{p}[[u_{1}]]} \Lambda_{\mathbb{F}_{p}[[u_{1}]]}HH_{1}^{\F_{p}[u_{1}]}\F_{p}[[u_{1}]] \otimes_{\F_{p}} \Lambda_{\F_{p}} dt_{1}.\]
\end{theorem}

\begin{proof}
From \eqref{HH of K(i)E2}, \cref{HH split} and \cref{HH(E2) flat over HH(E(2))}, $HH_{\ast}^{K(1)_{\ast}}(K(1)_{\ast}E_{2})$ is isomorphic to
$$(HH_{\ast}^{K(1)_{\ast}}K(1)_{\ast}E(2)\otimes_{\mathbb{F}_{p}[v_{1}][v_{2},v_{2}^{-1}]} \mathbb{F}_{p}[\mu_{p^{2}-1}][u_{1}][u, u^{-1}]) \otimes_{{F}_{p}[u_{1}]} 
\Lambda_{\mathbb{F}_{p}[[u_{1}]]}HH_{1}^{\F_{p}[u_{1}]}\F_{p}[[u_{1}]],$$
which from \cref{HH of E(2)} is isomorphic to
$$((K(1)_{\ast}E(2) \otimes \Lambda dt_{1}) \otimes_{\mathbb{F}_{p}[v_{1}][v_{2},v_{2}^{-1}]} \mathbb{F}_{p}[\mu_{p^{2}-1}][u_{1}][u, u^{-1}]) \otimes_{{F}_{p}[u_{1}]} 
\Lambda_{\mathbb{F}_{p}[[u_{1}]]}HH_{1}^{\F_{p}[u_{1}]}\F_{p}[[u_{1}]].$$
To use \eqref{K(i) of E2}, we rewrite this as
$$((K(1)_{\ast}E(2) \otimes \Lambda dt_{1}) \otimes_{\mathbb{F}_{p}[v_{1}][v_{2},v_{2}^{-1}]} \mathbb{F}_{p}[\mu_{p^{2}-1}][[u_{1}]][u, u^{-1}]) \otimes_{{F}_{p}[[u_{1}]]} 
\Lambda_{\mathbb{F}_{p}[[u_{1}]]}HH_{1}^{\F_{p}[u_{1}]}\F_{p}[[u_{1}]]$$ and then use the equation to get
$$HH_{\ast}^{K(1)_{\ast}}(K(1)_{\ast}E_{2}) \cong K(1)_{\ast}E_2 \otimes_{{F}_{p}[[u_{1}]]} \Lambda_{\mathbb{F}_{p}[[u_{1}]]}HH_{1}^{\F_{p}[u_{1}]}\F_{p}[[u_{1}]] \otimes_{\F_p} \Lambda_{\F_p} dt_{1}.$$
Here, elements of $K(1)_{\ast}E_{2}$ have homological degree $0$, $dt_{1}$ is as in \eqref{dt_1 class} and has homological degree $1$ and internal degree $2p^{2}-2$ and the elements of $HH_{1}^{\mathbb{F}_{p}[u_{1}]}\mathbb{F}_{p}[[u_{1}]]$ have homological degree $1$ and internal degree $0$.

As in the proof of \cref{K0 homology of THHE_2}, the multiplicative generators live in homological degree (column) $1$ and so the Bökstedt spectral sequence collapses on $\E^{2}$-page due to degree reasons. This is again a free graded commutative $K(1)_{\ast}E_2$ algebra on the first column, so there are no possible additive or multiplicative extensions and we have an isomorphism of $K(1)_{\ast}E_2$-algebras as in the statement of the result.
\end{proof}

\begin{theorem}\label{K(2) of THH(E2)}
The unit map \eqref{unit map} is a $K(2)$-equivalence, i.e., it induces an isomorphism of $K(2)_{\ast}E_{2}$-algebras $$K(2)_{\ast}E_{2} \cong K(2)_{\ast}THH(E_{2}).$$
\end{theorem}

\begin{proof}
As above, we use \eqref{HH of K(i)E2}, \cref{HH split} and \cref{HH(E2) flat over HH(E(2))}, to get $HH_{\ast}^{K(2)_{\ast}}(K(2)_{\ast}E_{2})$ isomorphic to
\[
   (K(2)_{\ast}E(2) \otimes_{\mathbb{F}_{p}[v_{1}][v_{2},v_{2}^{-1}]} \mathbb{F}_{p}[\mu_{p^{2}-1}][u_{1}][u, u^{-1}]) \otimes_{{\F}_{p}[u_{1}]} 
\Lambda_{\mathbb{F}_{p}[[u_{1}]]}HH_{1}^{\F_{p}[u_{1}]}\F_{p}[[u_{1}]].
\]
Extending scalars in the middle tensor factor from $\mathbb{F}_p[u_1]$ to $\mathbb{F}_{p}[[u_{1}]]$, we see that this is isomorphic to 
\[(K(2)_{\ast}E(2) \otimes_{\mathbb{F}_{p}[v_{1}][v_{2},v_{2}^{-1}]} \mathbb{F}_{p}[\mu_{p^{2}-1}][[u_{1}]][u, u^{-1}]) \otimes_{{\F}_{p}[[u_{1}]]} 
\Lambda_{\mathbb{F}_{p}[[u_{1}]]}HH_{1}^{\F_{p}[u_{1}]}\F_{p}[[u_{1}]].\]
Using \eqref{K(i) of E2} and \eqref{Mod p homotopy}, we can rewrite this as
\[
K(2)_{\ast}E_2 \otimes_{{\F}_{p}[[u_{1}]]} 
\Lambda_{\mathbb{F}_{p}[[u_{1}]]}HH_{1}^{\F_{p}[u_{1}]}\F_{p}[[u_{1}]].
\]
Using the formulas in \cite[B.5, Pg 167-171]{MR1192553}, we see that $\eta_R(v_1) = 0$ in $K(2)_*E_2$, and it follows that $u_1$ acts by $0$ on $K(2)_*E_2$. Since $HH_{1}^{\F_{p}[u_{1}]}\F_{p}[[u_{1}]]$ is a $[u_1,u_{1}^{-1}]$-module (above \eqref{field of fractions}), we have 
\[
K(2)_{\ast}E_2 \otimes_{{\F}_{p}[[u_{1}]]} 
\Lambda^n_{\mathbb{F}_{p}[[u_{1}]]}HH_{1}^{\F_{p}[u_{1}]}\F_{p}[[u_{1}]]=0
\]
for $n\geq 1$, and so
\[
K(2)_{\ast}E_2 \otimes_{{\F}_{p}[[u_{1}]]} 
\Lambda_{\mathbb{F}_{p}[[u_{1}]]}HH_{1}^{\F_{p}[u_{1}]}\F_{p}[[u_{1}]]
\cong
K(2)_{\ast}E_2 \otimes_{{\F}_{p}[[u_{1}]]} 
{\F}_{p}[[u_{1}]]
\cong K(2)_\ast E_2.  
\]
We conclude
$$HH_{\ast}^{K(2)_{\ast}}(K(2)_{\ast}E_{2}) \cong K(2)_{\ast}E_{2}.$$
In the Bökstedt spectral sequence, only the zeroth column is non-zero, and therefore the sequence collapses to give
\[K(2)_{\ast}THH(E_{2})\cong HH_{\ast}^{K(2)_{\ast}}(K(2)_{\ast}E_{2}) \cong K(2)_{\ast}E_{2}.\]
Since this map is induced by the unit map $E_{2} \rightarrow THH(E_{2})$, we conclude that the unit map is a $K(2)$-equivalence.
\end{proof}

\subsection{The case $i>2$: $K(i)_*THH(E_2)=0$} Since $E_2$ is $L_2$-local, for $i>2$
\begin{equation*}
    K(i)_\ast E_2 = 0.
\end{equation*}
Therefore, for $i>2$,
\begin{equation*}
    HH_{\ast, \ast}^{K(i)_{\ast}}K(i)_\ast E_2 = 0.
\end{equation*}
And from the Bökstedt spectral sequence, we get the result:
\begin{lemma}\label{higher i}
 For $i>2$, $K(i)_*THH(E_2)=0$. The spectrum $THH(E_2)$ is $L_2$-local.
\end{lemma}

\section{Lifting $K(i)_{\ast}THH(E_{2})$ classes to $\pi_{\ast}THH(E_{2})$ classes}\label{Lifting classes}

The purpose of this section is to lift certain $K(i)$-homology classes of $THH(E_2)$ to homotopy classes along the Hurewicz map 
$$\pi_{\ast}THH(E_{2}) \to K(i)_{\ast}THH(E_{2}).$$
We lift the 
classes represented by elements of $K(i)_{\ast}E_{2}$ in \cref{lift of unit}, the 
elements of $$HH_{1}^{\mathbb{F}_{p}[u_{1}]}\mathbb{F}_{p}[[u_{1}]] \in K(1)_{\ast}THH(E_{2})$$
in \cref{lift of Fp[[u1]]} and the 
class $dt_{1} \in K(1)_{\ast}THH(E_{2})$ in \cref{lift of dt}. 

To do the lifting, we use the commutative diagram below for $i = 0,1,2$. 
\begin{equation}\label{Commutative square}
\begin{tikzcd}
\pi_{\ast}(E_{2}) \otimes \pi_{\ast}(E_{2}) \arrow{r}{\theta} \arrow[swap]{d}{\Psi_{E_2}\otimes\Psi_{E_2}} & \pi_{\ast}(E_{2} \wedge E_{2}) \arrow{d}{\Psi_{E_2\wedge E_2}} \\
K(i)_{\ast}E_{2} \otimes_{K(i)_{\ast}} K(i)_{\ast}E_{2} \arrow{r}{\cong} & K(i)_{\ast}(E_{2} \wedge E_{2})
\end{tikzcd}
\end{equation}
The horizontal maps are the usual homology pairing maps, and the vertical maps are the Hurewicz maps.  Henceforth, we write just $\Psi$ for the Hurewicz maps when the source and target are clear.  The isomorphism in the diagram above is from \eqref{Künneth} and these isomorphic terms are column $1$ of $\E^{1}$-page of Bökstedt spectral sequence \eqref{BSS Equation}. Note also, that $\pi_{\ast}(E_{2}\wedge E_{2})$ is column $1$ of $\E^{1}$-page of filtration spectral sequence 
$$\E^{1}_{s,t} := \pi_{t}(E_{2}^{\wedge s+1}) \Longrightarrow \pi_{s+t}THH(E_{2})$$
arising from the simplicial filtration of $THH(E_2)$.
$\Psi$ extends to a map 
\begin{equation}\label{extension of psi}
    \pi_{t}(E_{2}^{\wedge s+1}) \to K(i)_{t}(E_{2}^{\wedge s+1})
\end{equation}
for each $s,t$ and it commutes with the respective differentials of the spectral sequences associated to each term, thus making it a map of $\E^1$-onwards pages of spectral sequences. Due to naturality, this map converges to the Hurewicz map $$\pi_{\ast}THH(E_{2}) \to K(i)_{\ast}THH(E_{2}).$$

\subsection{Lifting $K(i)_{\ast}E_{2}$}\label{lift of unit}
From \cref{K0 homology of THHE_2}, \cref{K(1) of THH(E2)}, \cref{K(2) of THH(E2)} $K(i)_{\ast}E_{2}$ is the zeroth homological degree term in $HH_{\ast}^{K(i)_{\ast}}(K(i)_{\ast}E_{2}) \cong K(i)_{\ast}THH(E_{2})$ for $i = 0,1,2$. From \cref{section unit map}, this is induced from the unit map which splits. The homological classes $K(i)_\ast E_2$ correspond exactly to the homotopy classes of $E_2$ in $\pi_\ast THH(E_2)$.

\subsection{Lifting the class $dt_{1}$ from $K(1)_{\ast}THH(E_{2})$}
From \eqref{t classes 1}, \eqref{t_1 class}, \eqref{degenerate class}, \eqref{dt_1 class}, we have 
$$t_{1} \in \pi_{2p-2}(E_2 \wedge E_2)$$
which maps along $\Psi$ to $$1 \otimes t_1 \in K(i)_{\ast}E_2 \otimes_{K(i)_{\ast}} K(i)_{\ast}E_2$$ and is denoted by $$dt_{1} \in K(i)_{2p-1}THH(E_{2}).$$

We show

\begin{proposition}\label{lift of dt}
There exists a class $\lambda_{1} \in THH_{2p-1}(E_{2})$ such that under Hurewicz homomorphism
$$\lambda_{1} \mapsto dt_{1} \in K(1)_{2p-1}THH(E_{2}).$$ Further, $$\Psi(\lambda_1) \in K(0)_{2p-1}THH(E_{2})$$ is a $K(0)_{\ast}E_{2}$-linear combination of classes $du$ and $du_{1}$ as in \eqref{du classes}.
\end{proposition}

\begin{proof}
    Using \eqref{Commutative square} and the discussion in \eqref{extension of psi}, for $i=1$, we have a commutative diagram
\[\begin{tikzcd}[row sep=large, column sep=large]
t_{1} \in \pi_{2p-2}(E_2 \wedge E_2) \ar[r, dotted] \ar[d, "\Psi"] & THH_{2p-1}(E_{2})\ar[d, "\Psi"]\\
1 \otimes t_{1} \in K(1)_{\ast}E_2 \otimes_{K(1)_{\ast}} K(1)_{\ast}E_2 \ar[r, dotted] & dt_{1} \in K(1)_{2p-1}THH(E_{2})
\end{tikzcd}\]
where the horizontal dotted arrows depict how surviving classes of a spectral sequence map from the first page of spectral sequence to the final one and then to a class of the object the spectral sequence converges to.
Since the diagram commutes, the image of $t_1$ along the top horizontal arrow is the required lift. We denote this class as $\lambda_{1}$ following \cite{MR1209233}. By the same argument applied to $i=0$, $\lambda_{1}$ also maps to the class $$dt_{1} \in K(0)_{2p-1}THH(E_{2}).$$ This class is related to the classes $$du, du_{1} \in K(0)_{\ast}THH(E_{2})$$ via \eqref{u,v,t}. We get $$du_{1} = d(u^{1-p}pt_{1}) = (1-p)pt_{1}u^{-p}du + pu^{1-p}dt_{1}$$
in $K(0)_{\ast}THH(E_{2})$ giving us the claimed linear combination.
\end{proof}

\subsection{Lifting the classes $HH_{1}^{\F_{p}[u_{1}]}\F_{p}[[u_{1}]]$ from $K(1)_{\ast}THH(E_{2})$} 
Due to \eqref{power series vs field of fractions}, we work with $$HH_{1}^{\F_{p}(u_{1})}\F_{p}((u_{1})).$$ From \cite[Example 3.3]{MR1853116}, this is an infinite dimensional vector space over $\F_{p}((u_{1}))$ with dimension the cardinality of the continuum. Let $\{s\}_I$ be a basis set of $\F_{p}((u_{1}))$ over $\F_{p}(u_{1})$, then we can choose a basis set for $$\F_{p}((u_{1})) \otimes_{\F_{p}(u_{1})} \F_{p}((u_{1}))$$ 
over $\F_{p}((u_{1}))$ (where the latter acts on former the way it would act on each level of the Hochschild complex, in our case on the left tensor factor) from the set $\{1 \otimes s\}_I$. Without loss of generality, we choose $\{s\}_I$ so that each $s \in \{s\}_I$ satisfies $$s \in \F_{p}[[u_{1}]]$$ and hence has (non-unique) lifts to $\Z[[u_{1}]]$ and $\Z_{p}[[u_{1}]]$. We denote such a lift of $s$ by $s'$. Further, $HH_{1}^{\F_{p}(u_{1})}\F_{p}((u_{1}))$ is a quotient of $\F_{p}((u_{1}))\otimes_{\F_{p}(u_{1})} \F_{p}((u_{1})),$
so we have a basis 
\begin{equation}\label{notation B}
    \mathfrak{B} \subset \{1 \otimes s\}_I
\end{equation}
of $HH_{1}^{\F_{p}(u_{1})}\F_{p}((u_{1}))$ over $\F_{p}((u_{1}))$. Choose a lift $s'$ for each $s$ such that $(1\otimes s) \in \B$. Then there are classes 
$$1 \otimes s' \in \pi_{0}E_{2} \otimes \pi_{0}E_{2}$$
since $$\pi_{0}E_{2} \otimes \pi_{0}E_{2} \cong \mathbb{Z}_{p}[\mu_{p^{2}-1}][[u_{1}]] \otimes \mathbb{Z}_{p}[\mu_{p^{2}-1}][[u_{1}]].$$
These classes map to $$\Psi(1 \otimes s') \in K(1)_{0}E_{2} \otimes_{K(1)_0} K(1)_{0}E_{2}$$ on $\E^1$-page of the Bökstedt spectral sequence \eqref{E1 page BSS} and further down to $[1 \otimes s] \in HH_{(1,0)}^{K(1)_{\ast}}K(1)_{\ast}E_{2}$ on $\E^2$-page. We lift each element of $\B$.

\begin{proposition}\label{lift of Fp[[u1]]}
For each $(1 \otimes s) \in \mathfrak{B}$, there exists a class $\tilde{s} \in THH_{1}E_{2}$, such that under Hurewicz homomorphism each $\tilde{s}$ maps to the corresponding class $$(1 \otimes s) \in K(1)_{1}THH(E_{2}).$$ 
\end{proposition}

\begin{proof}
From \eqref{Commutative square} and the discussion in \eqref{extension of psi}, we have a commutative diagram
\[\begin{tikzcd}[row sep=large, column sep=large]
\theta(1 \otimes s') \in \pi_{0}(E_2 \wedge E_2) \ar[r, dotted] \ar[d, "\Psi"] & THH_{1}(E_{2})\ar[d, "\Psi"]\\
\Psi(1 \otimes s') \in K(1)_{0}E_2 \otimes_{K(1)_{0}} K(1)_{0}E_2 \ar[r, dotted] & [1 \otimes s] \in K(1)_{1}THH(E_{2})
\end{tikzcd}\]
where, as before, the horizontal dotted arrows depict how surviving classes of a spectral sequence map from the first page of spectral sequence to the final one and to a class of the object the spectral sequence converges to. Since the class $$[1 \otimes s] \in K(1)_{1}THH(E_{2})$$ is non-zero, $1 \otimes s'$ survives the filtration spectral sequence representing a non-zero class in $THH_{1}(E_{2})$. We denote this class as $\tilde{s}$.
\end{proof}

\section{$THH(E_{2})$ and $THH(E_{2})_{p}^{\wedge}$}\label{Main results}

In this section we prove our main statements \cref{MainThm}, \cref{MainCorr}. This essentially amounts to constructing multiple cofiber sequences using the lifted classes from \cref{lift of unit}, \cref{lift of dt} and \cref{lift of Fp[[u1]]}. We do this in Subsection 5.1, Subsection 5.2 and Subsection 5.3, respectively.

To construct the cofiber sequences, we need the following remarks.

\begin{remark}\label{Properties of cofibers}
For $$A \rightarrow B \rightarrow C$$ a cofiber sequence of spectra with $A, B$ both $L_n$-local (\cref{Notations},\eqref{localization}), we have
\begin{enumerate}[leftmargin=*]
    \item The cofiber, $C$, is $L_{n}$-local.
    \item If $A \rightarrow B$ is a $K(n)$-equivalence, then $C$ is $L_{n-1}$-local.
\end{enumerate}
\end{remark}

\begin{remark}\label{M[u inverted]}
    In places in this section, we work in terms of $E_2[u_1^{-1}]$-modules: by \cite[\S VIII.4]{MR1417719}, we can invert the element $u_1$ in $\pi_0E_2$ to produce an $E_\infty$ ring spectrum $E_2[u_1^{-1}]$ with 
\[\pi_*(E_2[u_1^{-1}])\cong E_{2\ast}[u_1^{-1}].\] 
Likewise we can invert the action of $u_1$ on an $E_2$-module $M$ to produce an $E_2[u_1^{-1}]$-module $M[u_1^{-1}]$ with $$\pi_*(M[u_1^{-1}])\cong M_*[u_1^{-1}].$$ More generally, for any homology theory on $E_2$-modules, $(-)[u_1^{-1}]$ has the effect of inverting the action of $u_1$, and so in particular, we have
\[
K(i)_*(E_2[u_1^{-1}])\cong (K(i)_*E_2)[u_1^{-1}]\quad\text{and}\quad
K(i)_*(M[u_1^{-1}])\cong (K(i)_*M)[u_1^{-1}].
\]
Homotopically, the underlying $E_2$-module of $M[u_1^{-1}]$ is modeled as a homotopy colimit by multiplication by $u_1$:
\begin{equation}
    M[u_1^{-1}] \simeq \hocolim(M\xrightarrow{u_{1}}M\xrightarrow{u_{1}}M\xrightarrow{u_{1}}\cdots).
\end{equation}
Note that if $M$ is an $E_2$-module on which $u_1$ is already invertible, then as $E_2$-modules $$M\simeq M[u_1^{-1}];$$ in such a case (when working in the homotopy category of $E_2$-modules), we will regard $M$ as an $E_2[u_1^{-1}]$-module by abuse of notation.
\end{remark}

\subsection{The cofiber of unit map}
We denote the unit map from \eqref{unit map} as $f_1$ and consider the associated cofiber sequence 
\begin{equation}\label{f1,C_f1}
E_2 \xrightarrow{f_1} THH(E_2)\xrightarrow{C_{f_1}} \overline{THH}(E_{2}).
\end{equation}
Here we are writing $\overline{THH}(E_{2})$ for the cofiber of $f_1$, which we study in this subsection. We have two immediate observations from \cref{Properties of cofibers}, \cref{M[u inverted]}.

\begin{observation}\label{C_f1 is L_1}
    $E_2$ is $L_2$-local and from \cref{higher i}, $THH(E_{2})$ is $L_{2}$-local. Using \cref{Properties of cofibers}, we conclude $\overline{THH}(E_{2})$ is $L_{1}$-local.
\end{observation}  

\begin{observation}\label{u_1 inv module st}
The following isomorphism
\[K(1)_{\ast}E_{2} \otimes_{\F_{p}[[u_{1}]]} HH_{1}^{\F_{p}(u_{1})}\F_{p}((u_{1})) \cong K(1)_{\ast}E_{2}[u_{1}^{-1}] \otimes_{\F_{p}((u_{1}))} HH_{1}^{\F_{p}(u_{1})}\F_{p}((u_{1}))\]
implies that $\mathfrak{B}$ \eqref{notation B} is a basis of $$K(1)_{\ast}E_{2}[u_{1}^{-1}] \otimes_{{F}_{p}((u_{1}))} HH_{1}^{\F_{p}(u_{1})}\F_{p}((u_{1}))$$ as a $K(1)_{\ast}E_{2}[u_{1}^{-1}]$-module.
\end{observation} 

\begin{remark}\label{homological degree}
Note that since
\[K(i)_{\ast}THH(E_{2}) \cong HH_{\ast, \ast}^{K(i)_{\ast}}K(i)_{\ast}E_{2},\]
we can give $K(i)_{\ast}THH(E_{2})$ a grading using the extra homological grading on the right hand side.  We refer to this grading on $K(i)_{\ast}THH(E_{2})$ as the \textit{homological grading} and the degree of elements in this grading as their \textit{homological degree}.  In the case $i=0$ and $i=1$, \cref{K0 homology of THHE_2} and \cref{K(1) of THH(E2)} write the right hand side as an exterior algebra in certain classes of homological degree $1$ (that depend on $i$).  In terms of these classes, the homological grading is therefore a homogeneous grading in the intrinsic multiplication. In the case $i=2$, \cref{K(2) of THH(E2)} asserts that the homological grading is concentrated in degree zero.  
\end{remark}

We can now reorganize the computation of $K(i)_*THH(E_2)$ in terms of the homological grading of the previous remark.  The following theorems are just restatements of \cref{K(1) of THH(E2)} and \cref{K0 homology of THHE_2}. $\mathfrak{B}$ is as in \eqref{notation B}. The first result uses \cref{u_1 inv module st} to write the homological classes as internal direct sum, no similar analogue exists for the second result.

\begin{theorem}[\cref{K(1) of THH(E2)}]\label{K(1) classes}
In terms of the homological grading of \cref{homological degree}, the homological degree $n$ part of $K(1)_{\ast}THH(E_{2})$ is the $K(1)_{\ast}E_{2}$-module described as follows:
\indent
\begin{itemize}[leftmargin=*]
\item $K(1)_{\ast}E_{2}$ in homological degree $0$
\item The internal direct sum of the $K(1)_{\ast}E_{2}$-modules $K(1)_{\ast}E_{2}(dt_{1})$ and $K(1)_{\ast}E_{2}[u_{1}^{-1}](1 \otimes s)$, for all $(1 \otimes s) \in \mathfrak{B}$,  in homological degree 1
\item The internal direct sum of the $K(1)_{\ast}E_{2}$-modules $$K(1)_{\ast}E_{2}[u_{1}^{-1}](1 \otimes s_{1}) (1 \otimes s_{2}) \cdots (1 \otimes s_{n})$$ $($for all choices of $n$ distinct elements $(1 \otimes s_i)\in \mathfrak{B})$ and $$K(1)_{\ast}E_{2}[u_{1}^{-1}](1 \otimes s_{1}) (1 \otimes s_{2}) \cdots (1 \otimes s_{n-1}) dt_{1}$$ $($for all choices of $n-1$ distinct elements $(1 \otimes s_i)\in \mathfrak{B})$, in homological degree $n$, $n \geq 2$.
\end{itemize}
\end{theorem}

\begin{theorem}[\cref{K0 homology of THHE_2}]\label{K(0) classes}
In terms of the homological grading of \cref{homological degree}, the homological degree $n$ part of $K(0)_{\ast}THH(E_{2})$ is the $K(0)_{\ast}E_{2}$-module described as follows:
\indent
\begin{itemize}[leftmargin=*]
\item $K(0)_{\ast}E_{2}$ in homological degree $0$,
\item The internal direct sum of $K(0)_{\ast}E_{2}(du)$ and $K(0)_{\ast}E_{2} \otimes_{\X}HH_{1}^{\Q}\X$ in homological degree $1$,
\item $K(0)_{\ast}E_2$ module generated by products of $n$ distinct degree $1$ classes in homological degree $n$.
\end{itemize}
\end{theorem}

This allows us to know the homological classes of $\overline{THH}(E_{2})$.
\begin{observation}\label{homological classes of C_f1}
    From \cref{C_f1 is L_1}, $\overline{THH}(E_{2})$ has only $K(0)$ and $K(1)$-homology classes. $f_{1}$ maps on to homological degree $0$ classes in $K(i)_{\ast}THH(E_2)$ for $i=0,1$ from \cref{K(1) classes}, \cref{K(0) classes}. The remaining classes from \cref{K(0) classes} and \cref{K(1) classes} correspond to respective $K(0)$ and $K(1)$-homology classes of $\overline{THH}(E_{2})$, respectively. 
\end{observation}

\subsection{Second cofiber sequence}
Next we use the analysis of $\overline{THH}(E_{2})$ above to construct a second cofiber sequence of $E_2$-modules
\begin{equation*}\label{eq:second cofiber sequence}
    X_2 \xrightarrow[]{f_2} \overline{THH}(E_{2}) \xrightarrow[]{C_{f_2}} \bbTHH(E_{2})
\end{equation*}
where $C_{f_2}$ is the cofiber map and $\bbTHH(E_2)$ denotes the cofiber. Then, we analyze $\bbTHH(E_2)$.

From \cref{lift of dt}, we have 
\begin{equation*}
  \lambda_{1}: \bbS^{2p-1} \to THH(E_{2}).  
\end{equation*}
Smashing with $E_{2}$, using the $E_{2}$-module structure of $THH(E_{2})$, and composing with $C_{f_1}$, we get $E_2$-module map
\[ j_{1} : \Sigma^{2p-1}E_{2} \simeq E_{2} \wedge \bbS^{2p-1} \rightarrow E_{2} \wedge THH(E_{2}) \rightarrow THH(E_{2}) \rightarrow \overline{THH}(E_{2}).\]
Since $\overline{THH}(E_{2})$ is $L_{1}$-local, $j_{1}$ factors through the following map of $E_2$-modules
$$\overline{j_{1}}: \Sigma^{2p-1}L_{1}E_{2} \rightarrow \overline{THH}(E_{2}).$$
By construction of $j_{1}$, the induced map at $K(1)_{\ast}$ level sends the generator of $K(1)_{\ast}(\Sigma^{2p-1}E_{2})$ to the class 
\begin{equation}\label{dt_1 of C_f1}
    dt_{1} \in K(1)_{\ast}\overline{THH}(E_{2})
\end{equation}
and hence the same is true for the map $\overline{j_{1}}$. Let
\begin{equation}\label{Define X_2 and f_2}
    X_{2}:= \Sigma^{2p-1}L_{1}E_{2}, \ f_{2}:= \overline{j_{1}}: \Sigma^{2p-1}L_{1}E_{2} \rightarrow \overline{THH}(E_{2}).
\end{equation}

\begin{observation}\label{homological classes of C_f2}
    Similar to \cref{homological classes of C_f1}, from \eqref{Define X_2 and f_2} we conclude that $\bbTHH(E_2)$ is $L_1$-local and its $K(1)$-homology classes are the ones generated by all classes of $K(1)_{\ast}\bTHH(E_2)$ except \eqref{dt_1 of C_f1}. From \cref{u_1 inv module st} and \cref{K(1) classes}, $K(1)_{\ast}\bbTHH(E_2)$ is now a $K(1)_{\ast}E_{2}[u_{1}^{-1}]$-module. 
\end{observation} 

\subsection{Third cofiber sequence}
There is a natural map
\begin{equation}\label{inverting u}
    \bbTHH(E_2) \rightarrow \bbTHH(E_2)[u_{1}^{-1}].
\end{equation}
From \cref{M[u inverted]}, \cref{homological classes of C_f2}
\begin{equation}\label{K(1)equiv}
    K(1)_{\ast}(\bbTHH(E_2)[u^{-1}_{1}]) \cong (K(1)_{\ast}\bbTHH(E_2))[u_{1}^{-1}] \cong K(1)_{\ast}\bbTHH(E_2).
\end{equation}
From $K(1)$-equivalence and $L_1$-locality of 
\begin{equation*}
  \bbTHH(E_2)[u_{1}^{-1}]  \text{ and } \bbTHH(E_2)
\end{equation*}
we conclude, the cofiber of \eqref{inverting u} is rational using \cref{Properties of cofibers}. Let $C_{\infty}$ denote this cofiber.

Consider $\lambda_{1}$ and $\tilde{s}$ from \cref{lift of dt} and \cref{lift of Fp[[u1]]}. Let $\tilde{\mathfrak{B}}$ denote the chosen lifts $\tilde{s}$ of elements $(1 \otimes s) \in \mathfrak{B}$ from \cref{lift of Fp[[u1]]}. We denote by
\begin{equation}\label{q_alpha}
 q_{\alpha}: \bbS^{|\alpha|} \rightarrow THH(E_{2}),   
\end{equation}
the elements of the form $$\Pi_{\alpha}\tilde{s} \in \pi_{|\alpha|}THH(E_{2}),$$ where $\tilde{s}$ are all distinct elements of set $\alpha$, $\alpha$ ranges over all non-empty finite subsets of $\tilde{\mathfrak{B}}$ and $|\alpha|$ denotes the cardinality of these subsets. Similarly, let
\begin{equation}\label{q_alpha,lambda}
 q_{\alpha, \lambda_{1}}: \bbS^{|\alpha|+2p-1} \rightarrow THH(E_{2})   
\end{equation}
denote elements 
$$\Pi_{\alpha}\tilde{s}\cdot\lambda_{1} \in \pi_{|\alpha|+2p-1}THH(E_{2})$$
for all $\alpha$ as above.

Composing \eqref{q_alpha}, \eqref{q_alpha,lambda} with $C_{f_{1}}$, $C_{f_{2}}$ and \eqref{inverting u} we get maps from suspensions of sphere spectrum to $\bbTHH(E_2)[u_{1}^{-1}]$ which is an $E_{2}[u_{1}^{-1}]$-module. Smashing these maps with $E_{2}[u_{1}^{-1}]$ and using the $E_{2}[u_{1}^{-1}]$-module structure of $\bbTHH(E_2)[u_{1}^{-1}]$, we get $E_2[u_{1}^{-1}]$-module maps
\begin{equation}
    \begin{aligned}
       &\overline{q_{\alpha}} : \Sigma^{|\alpha|}E_{2}[u_{1}^{-1}] \simeq E_{2}[u_{1}^{-1}] \wedge \bbS^{|\alpha|} \rightarrow \bbTHH(E_2)[u_{1}^{-1}] \\
       &\overline{q_{\alpha, \lambda_{1}}} : \Sigma^{|\alpha|+2p-1}E_{2}[u_{1}^{-1}] \simeq E_{2}[u_{1}^{-1}] \wedge \bbS^{|\alpha|+2p-1} \rightarrow \bbTHH(E_2)[u_{1}^{-1}].
    \end{aligned}
\end{equation}
If we repeat the above process without composing with \eqref{inverting u} and using the $E_{2}$-module structure of $\bbTHH(E_2)$, we have $E_2$-module maps
\begin{equation}
    \begin{aligned}
       &\widetilde{q_{\alpha}} : \Sigma^{|\alpha|}E_{2} \simeq E_{2} \wedge \bbS^{|\alpha|} \rightarrow E_{2} \wedge \bbTHH(E_2) \rightarrow \bbTHH(E_2)\\ 
       &\widetilde{q_{\alpha, \lambda_{1}}} : \Sigma^{|\alpha|+2p-1}E_{2} \simeq E_{2} \wedge \bbS^{|\alpha|+2p-1} \rightarrow E_{2} \wedge \bbTHH(E_2) \rightarrow \bbTHH(E_2).
    \end{aligned}
\end{equation}
Note $\overline{q_{\alpha}}$, $\overline{q_{\alpha, \lambda_{1}}}$ are $u_{1}$ localizations of $\widetilde{q_{\alpha}}$ and $\widetilde{q_{\alpha, \lambda_{1}}}$, i.e., $\overline{q_{\alpha}} = \widetilde{q_{\alpha}}[u_{1}^{-1}]$ and $\overline{q_{\alpha, \lambda_{1}}} = \widetilde{q_{\alpha, \lambda_{1}}}[u_{1}^{-1}]$. From $L_{1}$-locality of $\bbTHH(E_2)$ and $\bbTHH(E_2)[u_{1}^{-1}]$ these maps factor through the following $E_2$-module and $E_2[u_1^{-1}]$-module maps, respectively:
\begin{equation*}
    \begin{aligned}
        \widetilde{q_{\alpha}}': \Sigma^{|\alpha|}L_{1}E_{2} \rightarrow \bbTHH(E_2), & \ \widetilde{q_{\alpha, \lambda_{1}}}': \Sigma^{|\alpha|+2p-1}L_{1}E_{2} \rightarrow \bbTHH(E_2)\\
        \overline{q_{\alpha}}': \Sigma^{|\alpha|}L_{1}E_{2}[u_{1}^{-1}] \rightarrow \bbTHH(E_2)[u_{1}^{-1}], & \ \overline{q_{\alpha, \lambda_{1}}}': \Sigma^{|\alpha|+2p-1}L_{1}E_{2}[u_{1}^{-1}] \rightarrow \bbTHH(E_2)[u_{1}^{-1}].
    \end{aligned}
\end{equation*}
We are ready to define the third cofiber sequence.
\begin{construction}\label{Q,X_3,f_3}
  Let $\widetilde{Q}$ be the map of $E_2$-modules
\[\bigvee_{\alpha} (\widetilde{q_{\alpha}}' \bigvee \widetilde{q_{\alpha, \lambda_{1}}}'): \bigvee_{\alpha} (\Sigma^{|\alpha|}L_{1}E_{2} \bigvee \Sigma^{|\alpha|+2p-1}L_{1}E_{2}) \rightarrow \bbTHH(E_2)\]
and $\overline{Q}$ be the map of $E_2[u_1^{-1}]$-modules
\[\overline{Q}: \bigvee_{\alpha} (\overline{q_{\alpha}}' \bigvee \overline{q_{\alpha, \lambda_{1}}}'): \bigvee_{\alpha} (\Sigma^{|\alpha|}L_{1}E_{2}[u_{1}^{-1}] \bigvee \Sigma^{|\alpha|+2p-1}L_{1}E_{2}[u_{1}^{-1}]) \rightarrow \bbTHH(E_2)[u_{1}^{-1}].\]
Further, define $$X_{3} := \bigvee_{\alpha} (\Sigma^{|\alpha|}L_{1}E_{2} \bigvee \Sigma^{|\alpha|+2p-1}L_{1}E_{2}).$$
Note that this makes
\[\bigvee_{\alpha} (\Sigma^{|\alpha|}L_{1}E_{2}[u_{1}^{-1}] \bigvee \Sigma^{|\alpha|+2p-1}L_{1}E_{2}[u_{1}^{-1}]) = X_{3}[u_{1}^{-1}].\] Then our third cofiber sequence is given by
\begin{equation}\label{define X_3 and f_3}
    X_3 \xrightarrow[]{f_3:= \widetilde{Q}} \bbTHH(E_2).
\end{equation}
\end{construction}
We have the following comparision of $\widetilde{Q}$ and $\overline{Q}$:

\begin{theorem} \label{rational cofiber}
We have a commutative diagram as follows
\begin{equation}\label{eq:rational cofiber}
\begin{tikzcd}
X_{3} \arrow{d}{\widetilde{Q}} \arrow{r}
& X_{3}[u_{1}^{-1}] \arrow{d}{\overline{Q}}
& \\
\bbTHH(E_2) \arrow{r} & \bbTHH(E_2)[u_{1}^{-1}] \arrow{d} \arrow{r} & C_{\infty}
& \\
                                   & C_{\overline{Q}}
\end{tikzcd}
\end{equation}
where $C_{\overline{Q}}$ denotes the cofiber of $\overline{Q}$. Moreover, the map $\overline{Q}$ is a $K(1)_{\ast}$-isomorphism and both $C_{\infty}$ and $C_{\overline{Q}}$ are rational.
\end{theorem}

\begin{proof}
The commutativity of the diagram and rationality of $C_{\infty}$ follows from their construction in this section. All the terms in the square part of the diagram are $L_1$-local and hence rationality of $C_{\overline{Q}}$ follows from \cref{Properties of cofibers} once we show $\overline{Q}$ induces a $K(1)_{\ast}$-isomorphism.

By construction, $q_{\alpha}$ and $q_{\alpha, \lambda_{1}}$ are the lifts of $K(1)$-homology classes of $THH(E_2)$ to $\pi_{\ast}THH(E_{2})$. Using \cref{homological classes of C_f2} and \eqref{K(1)equiv}, they can be extended to lifts of $K(1)$-homology classes of $\bbTHH(E_2)[u_1^{-1}]$ to $\pi_{\ast}\bbTHH(E_{2})[u_1^{-1}]$. Thus after taking $K(1)_{\ast}$ on each side of $\overline{Q}$ we have a wedge of 
\begin{align*}
& \Sigma^{|\alpha|}K(1)_{\ast}(L_{1}E_{2}[u_{1}^{-1}]) \cong K(1)_{\ast}(L_{1}E_{2}[u_{1}^{-1}]) \wedge \bbS^{|\alpha|} \cong K(1)_{\ast}E_{2}[u_{1}^{-1}] \wedge \bbS^{|\alpha|} \text{ and }\\
& \Sigma^{|\alpha|+2p-1}K(1)_{\ast}(L_{1}E_{2}[u_{1}^{-1}]) \cong K(1)_{\ast}(L_{1}E_{2}[u_{1}^{-1}]) \wedge \bbS^{|\alpha|+2p-1} \cong K(1)_{\ast}E_{2}[u_{1}^{-1}] \wedge \bbS^{|\alpha|+2p-1} 
\end{align*}
on the source side and under $K(1)_{\ast}\overline{Q}$ they map to the wedge of
$$K(1)_{\ast}E_{2}[u_{1}^{-1}]\Pi_{\alpha}(1 \otimes s) \text{ and } K(1)_{\ast}E_{2}[u_{1}^{-1}]\Pi_{\alpha}(1 \otimes s) dt_{1}$$
which is the target, giving us the desired equivalence.
\end{proof}

\subsection{Proofs of the main results} 
From \eqref{f1,C_f1}, \eqref{Define X_2 and f_2}, \ref{Q,X_3,f_3} and \cref{rational cofiber} we have the following result.
\begin{theorem}\label{MainThm}
There is a diagram of $E_2$-modules
\begin{equation}\label{THH(E2)}
\begin{tikzcd}
E_{2} \arrow{r}{f_{1}} & THH(E_{2}) \arrow{d}{C_{f_{1}}}\\ 
X_{2} \arrow{r}{f_{2}} & \overline{THH}(E_{2}) \arrow{d}{C_{f_{2}}}\\
X_{3} \arrow{r}{f_{3}} & \bbTHH(E_2)
\end{tikzcd}
\end{equation}
such that
\begin{align*}
    E_2 \xrightarrow[]{f_1} & THH(E_2) \xrightarrow[]{C_{f_1}} \bTHH(E_2)\\
    X_{2} \xrightarrow{f_2} & \overline{THH}(E_{2}) \xrightarrow{C_{f_{2}}} \bbTHH(E_2)
\end{align*}
are cofiber sequences, and the map induced by $f_{3}$ $$X_{3}[u_{1}^{-1}] \rightarrow \bbTHH(E_2)[u_{1}^{-1}]$$ is a $K(1)_{\ast}$-isomorphism with a rational cofiber. Further, $X_{2}$ and $X_{3}$ are explicitly identifiable in terms of suspensions and localizations of $E_{2}$, and $X_{2}$, $X_{3}$, $\bTHH(E_2)$ and $\bbTHH(E_2)$ are all $L_{1}$-local. 
\end{theorem} 
 
Upon $p$-completion, the rational cofiber vanishes. We show for $p$-complete $THH$ of $E_{2}$ spectrum there is a complete description in terms of cofiber sequences made of terms that are suspensions and localizations of $E_{2}$ (up to $p$-completion):

\begin{theorem}\label{MainCorr}
We have the following diagram of $(E_2)^{\wedge}_p$-modules for $THH(E_{2})_{p}^{\wedge}$, where $(C_{f_{i}})_{p}^{\wedge}$ are the cofiber maps of $(f_{i})_{p}^{\wedge}$
\[\begin{tikzcd}
(E_{2})_{p}^{\wedge} \arrow{r}{(f_{1})_{p}^{\wedge}:= \text{p-completed unit map}} & THH(E_{2})_{p}^{\wedge} \arrow{d}{(C_{f_{1}})_{p}^{\wedge}}\\
\Sigma^{2p-1}L_{1}(E_{2})_{p}^{\wedge} \arrow{r}{(f_{2})_{p}^{\wedge}:=(\overline{\jmath_{1}})_{p}^{\wedge}} & \overline{THH}(E_{2})_{p}^{\wedge} \arrow{d}{(C_{f_{2}})_{p}^{\wedge}}\\
(\bigvee_{\alpha} (\Sigma^{|\alpha|}L_{1}E_{2}[u_{1}^{-1}] \bigvee \Sigma^{|\alpha|+2p-1}L_{1}E_{2}[u_{1}^{-1}]))_{p}^{\wedge}\arrow[r, phantom, sloped, "\simeq"]&\bbTHH(E_2)_{p}^{\wedge}
\end{tikzcd}
\]
\end{theorem}

\begin{proof}
This is mostly just the $p$-completion of the diagram in \cref{MainThm}. The only thing different is the equivalence. We get this by taking the $p$-completion of \eqref{eq:rational cofiber}. Since $C_\infty$ and $C_{\overline{Q}}$ are rational, their $p$-completion is a point. We get a commutative square
\[\begin{tikzcd}
(\bigvee_{\alpha}(\Sigma^{|\alpha|}L_{1}E_{2} \vee \Sigma^{|\alpha|+2p-1}L_{1}E_{2}))_{p}^{\wedge} \arrow{d}{\widetilde{Q}_{p}^{\wedge}} \arrow{r}
& (\bigvee_{\alpha}(\Sigma^{|\alpha|}L_{1}E_{2}[u_{1}^{-1}] \vee \Sigma^{|\alpha|+2p-1}L_{1}E_{2}[u_{1}^{-1}]))_{p}^{\wedge} \arrow{d}{\overline{Q}_{p}^{\wedge}}[swap]{\simeq}\\
\bbTHH(E_2)_{p}^{\wedge} \arrow{r}{\simeq} & \bbTHH(E_2)[u_{1}^{-1}]_{p}^{\wedge}
\end{tikzcd}\]
giving us the required equivalence.
\end{proof}


\bibliographystyle{alpha}
\bibliography{mybib}

\end{document}